\documentclass[12pt,a4paper, oneside]{report}
\usepackage[left=3.5cm, right=3.5cm]{geometry}

\usepackage[hyphens]{url}
\usepackage[a-1b]{pdfx}
\hypersetup{bookmarksnumbered = true}

\usepackage[english]{babel}

\usepackage[latin1]{inputenc}

\usepackage[T1]{fontenc}

\usepackage{graphicx}

\usepackage{enumerate}

\usepackage{amsmath, amssymb, amsthm, dsfont}
\usepackage{yfonts} 
\usepackage{tikz-cd} 
\usepackage{adjustbox}
\usepackage{xfrac}

\usepackage[
   style=alphabetic,
   backend=biber,
   language=german,
   giveninits,
   uniquename=init
]{biblatex}
\usepackage{csquotes}
\DeclareFieldFormat{urldate}{%
    (Last accessed %
    \thefield{urlday}\addperiod %
    0\thefield{urlmonth}\addperiod    %
    \thefield{urlyear}\isdot)}

\usepackage{lscape}
\usepackage{booktabs}
\usepackage{float}

\usepackage{makeidx}
\makeindex

\usepackage{mathtools}

\theoremstyle{definition}
\newtheorem{definition}{Definition}[section]
\newtheorem{definition and proposition}[definition]{Definition and Proposition}
\newtheorem*{defprop}{Definition and Proposition}
\newtheorem*{defin}{Definition}
\newtheorem{remark}[definition]{Remark}
\newtheorem{warning}[definition]{Warning}
\newtheorem{example}[definition]{Example}
\newtheorem{reminder}[definition]{Reminder}
\theoremstyle{plain}
\newtheorem{theorem}[definition]{Theorem}
\newtheorem*{theo}{Theorem}

\theoremstyle{plain}
\newtheorem{proposition}[definition]{Proposition}
\newtheorem*{prop}{Proposition}

\newtheorem{corollary}[definition]{Corollary}
\newtheorem{lemma}[definition]{Lemma}

\newcommand {\Fun}{
	\mathrm{Fun}}
\newcommand{\Pro}{\mathcal{P}\mathrm{ro}}
\newcommand {\Cond }{
	\mathcal{C}\mathrm{ond(}\mathcal{S}\mathrm{et)}}
\newcommand {\cat }{
	\mathcal{C}}
\newcommand {\Hom }{
	\mathrm{Hom}}
\newcommand {\colim }{
	\mathrm{colim}}
\newcommand {\Top}{
	\mathcal{T}\mathrm{op}}
\newcommand {\Set}{
	\mathcal{S}\mathrm{et}}
\newcommand{\Grp}{
	\mathcal{G}\mathrm{rp}}
\newcommand{\AbGrp}{
	\mathcal{A}\mathrm{b}\mathcal{G}\mathrm{rp}}
\newcommand {\Ani}{
	\mathcal{A}\mathrm{ni}}
\newcommand {\CoHa}{
	\mathcal{CH}\mathrm{aus}}
\newcommand {\SFin}{\mathcal{S}\mathrm{et^{fin}}}
	\newcommand {\PFin}{\mathcal{P}\mathrm{ro}\mathcal{F}\mathrm{in}}
\newcommand {\Profin}{\mathcal{P}\mathrm{ro}(\mathcal{S}\mathrm{et^{fin}})}
\newcommand {\extdis}{\mathcal{E}\mathrm{xtr}\mathcal{D}\mathrm{isc}}
\newcommand {\CondAni}{
	\mathcal{C}\mathrm{ond(}\Ani\mathrm{)}}
\newcommand {\AniCond}{
	\mathcal{A}\mathrm{ni(}\Cond\mathrm{)}}
\newcommand {\Kan}{
	\mathcal{K}\mathrm{an}}
\newcommand {\sSet}{
	\mathrm{s}\mathcal{S}\mathrm{et}}
\newcommand {\CW}{
	\mathcal{CW}}
\newcommand {\AniC}{
	\mathcal{A}\mathrm{ni(}\cat\mathrm{)}}

\begin{document}
\pagestyle{empty}
\begin{titlepage}
  \begin{center}
    
    

    \huge{Animated Condensed Sets and Their Homotopy Groups}
     \medskip
    
    \large{Catrin Mair}\\

   \large{May 17, 2021}
   
  \end{center}
 \bigskip
\noindent 
\textbf{Abstract.}
The theory of condensed mathematics by Dustin Clausen and Peter Scholze claims that topological spaces should be replaced by the definition of condensed sets. The main purpose of this paper is to investigate in which way the theory of homotopy groups on topological spaces can be extended to the theory of condensed sets.\\
We show that there exist functors from the category of condensed sets to the category of pro-groups, s.t. restriction to CW-spaces coincides with the ordinary notion of homotopy groups on topological spaces. These functors can be extended to the $\infty$-categories of condensed anima and pro-anima. 
On our way to these results we will prove that the compact projective objects in the category of condensed sets are given by the extremally disconnected profinite sets.\\
\medskip
\\
\textbf{MSC.} 14F35, 18F10, 18F20, 18N60, 54A05, 54G05, 55P65, 55Q05

\end{titlepage}

%
\tableofcontents
\pagestyle{headings}
\chapter*{Introduction}\index{Introduction}
\addcontentsline{toc}{chapter}{Introduction}\index{Introduction}
The theory of \textit{Condensed sets} was started to be developed by Dustin Clausen and Peter Scholze in 2018 and 2019, whereas the first definition already occurred 2013 in the paper \cite{Bhatt2014} by Bhargav Bhatt and Peter Scholze. It is closely related to the theory of \textit{Pyknotic sets} by Clark Barwick and Peter Haine as both theories aim to provide a setting in which working with algebraic objects that carry a topology is possible without trouble\footnote{This master thesis is mainly based on the theory by Clausen and Scholze \cite{Scholze2019, Scholze2019a, Scholze2019b, Clausen} but is also oriented towards the work of Barwick and Haine \cite{Barwick2018, Barwick2019, Haine2020, Haine2020a}.}. Moreover, Scholze states in his guest post on the \textit{Xena Project} \cite{Scholze2020} that the condensed formalism "can be fruitfully applied to real functional analysis". Nevertheless, this master thesis places emphasis on another direction of the notion of condensed sets, namely viewing them as replacement for and as improvement of topological spaces.
\begin{defin}{(cf.\ Prop.\ \ref{condsetequi})}
Let $\extdis$ be the category of extremally disconnected compact Hausdorff spaces, or equivalently, of extremally disconnected profinite sets (cf.\ Def.\ \ref{extdisdef}). A \textit{condensed set} is a contravariant functor
\begin{align*}
T \colon \extdis^{\mathrm{op}} &\rightarrow \Set\\
S&\mapsto T(S)
\end{align*}
satisfying $T(\emptyset)=\ast$ and sending finite disjoint unions to finite products.
\end{defin}
\begin{prop}{(cf.\ Ex.\ \ref{extop}, Prop.\ \ref{equivcomphaus})}
Every (sufficiently nice) topological space T can be considered as condensed set by \begin{align*}
\underline{T}\colon\extdis^{\mathrm{op}} &\rightarrow \Set\\
S &\mapsto \Hom_{\Top}(S,T)
\end{align*}
and one obtains a fully faithful embedding $\CW \hookrightarrow \Cond$ of the category $\CW$ of CW-spaces into the category $\Cond$ of condensed sets.
\end{prop}
\section*{Aim and structure}
The ultimate goal that we pursue is to examine a notion of \textit{homotopy groups} for condensed sets. As the class of topological spaces that are of most interest within algebraic topology, the CW-spaces, embeds fully faithfully into condensed sets, there are high hopes to develop a good notion. We are following an idea by Peter Scholze communicated via a short email conversation in September 2020.\\
\newline
The following schematic diagram, where just the small right triangle is commutative, provides an outline of the structure of this master thesis indicating central $\infty$-categories and connecting functors which will be discussed in the corresponding chapters.
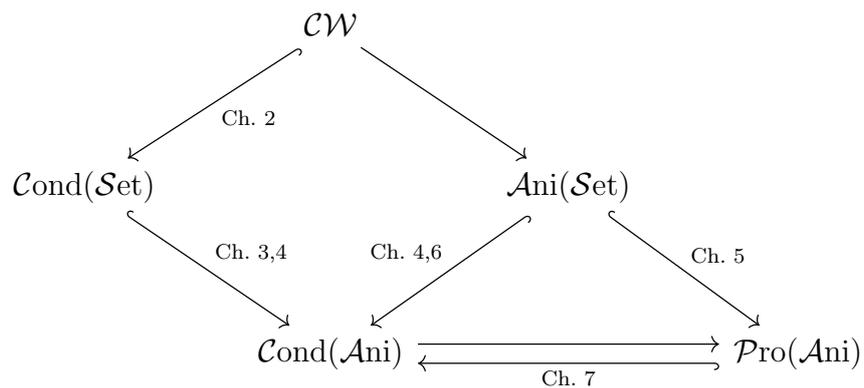
\begin{equation*}
\begin{tikzcd}[row sep=huge]
&
\CW  \arrow[hookrightarrow, dl, "\text{Ch.\ } \ref{condset}"] 
 \arrow[dr] & & 
\\
\Cond \arrow[hookrightarrow,dr,"\text{Ch.\ } \ref{compro}\text{,}\ref{animation}"] & & \Ani(\Set)\arrow[hookrightarrow,dl,swap,"\text{Ch.\ } \ref{animation}\text{,}\ref{condani}"] \arrow[hookrightarrow,dr, "\text{Ch.\ } \ref{procatbase}"] & 
\\ 
& \CondAni \arrow[rr, yshift=0.7ex] & & \Pro(\Ani) \arrow[hookrightarrow, ll, yshift=-0.7ex, "\text{Ch.\ } \ref{homotopy}"]
\end{tikzcd}
\end{equation*}
\section*{Overview of important results}
Besides the fully faithful embedding $\CW \hookrightarrow \Cond$, one can consider a CW-space up to (weak) homotopy equivalence as a Kan complex by the singular simplicial complex construction. This leads to a (non-faithful) functor $\CW \rightarrow \Ani(\Set)$, where $\Ani(\Set)$ or $\Ani$ (Ch.\ \ref{animation}) denotes the $\infty$-category obtained by inverting weak homotopy equivalences of topological spaces or simplicial sets, equivalently. In the terminology of Scholze and Clausen, its objects are called \textit{anima}\footnote{In the pyknotic setting, they are called \textit{spaces}.}. The category of condensed sets $\Cond$ (Ch.\ \ref{condset}) and the $\infty$-category $\Ani(\Set)$ again have fully faithful embeddings $\Cond \hookrightarrow \CondAni$ and $\Ani(\Set) \hookrightarrow \CondAni$ in a common generalized $\infty$-category $\CondAni$ (Ch.\ \ref{condani}) which combines the topological space direction and the homotopy theory direction. The objects of this $\infty$-category are, on behalf of the claimed combination, called \textit{condensed anima}\footnote{The corresponding notion in the pyknotic setting is the $\infty$-category of \textit{pyknotic spaces}.}. The following description of condensed anima represents how the $\infty$-category $\CondAni$ of condensed anima can be seen as the higher categorical analogue of $\Cond$.
\begin{defin}{(cf.\ Def.\ \ref{condanidef})}
We identify the category $\extdis$ with its $\infty$-category obtained by the nerve construction. A \textit{condensed anima} is a contravariant $\infty$-functor
\begin{align*}
T \colon \extdis^{\mathrm{op}} &\rightarrow \Ani\\
S&\mapsto T(S)
\end{align*}
taking finite disjoint unions to finite products.
\end{defin}
Before coming to this definition, the $\infty$-category $\CondAni$ will arise from another point of view in this master thesis. Namely, the $\infty$-categories $\Ani$ and $\CondAni$ can be obtained by a general process called \textit{animation} (Ch.\ \ref{animation}). 
\begin{defprop}{(cf.\ Def.\ \ref{defani}, Prop.\ \ref{truncationprop2})}
Let $\cat$ be an ordinary cocomplete category generated under small colimits by its subcategory of compact projective objects $\cat^{\mathrm{cp}}$. The \textit{animation} of $\cat$ is defined as the $\infty$-category $\Ani(\cat)$ freely generated under sifted colimits by $\cat^{\mathrm{cp}}$. One has a fully faithful embedding $\cat \hookrightarrow \Ani(\cat)$.
\end{defprop}
The $\infty$-category $\Ani(\Set)$, respectively $\CondAni$, is (equivalent to) the animation of $\Set$, respectively $\Cond$ (cf.\ Def.\ \ref{aniset}, Def.\ \ref{anicond}, Thm.\ \ref{thmAniCond})\footnote{The original source for these results is Chapter 11, escpecially Example 11.6., of \cite{Scholze2019a}.}. Our main purpose of the first half of this master thesis will be to determine the compact projective objects of $\Cond$ in order to define $\Ani(\Cond)$ and to obtain the fully faithful embedding $\Cond \hookrightarrow \Ani(\Cond)\cong\CondAni$.
\begin{theo}{(cf.\ Thm.\ \ref{thmcompro})}
The subcategory $\Cond^{\mathrm{cp}}$ of compact projective objects in the category $\Cond$ of condensed sets is given by the category $\extdis$ of extremally disconnected compact Hausdorff spaces.
\end{theo}
We will see that the fully faithful embedding $\Cond \hookrightarrow \CondAni$ is given by composition of a functor in $\Cond$, viewed as $\infty$-functor, with the fully faithful embedding $\Set \hookrightarrow \Ani(\Set)$ to obtain an $\infty$-functor in $\CondAni$ (cf.\ Thm.\ \ref{thmAniCond}). In the same theorem, we will prove that the fully faithful embedding $\Ani \hookrightarrow \CondAni$ is given by the so called \textit{discrete functor} $\underline{(\cdot)}^{\mathrm{disc}}$, which is induced by our key construction of Chapter \ref{condani}: The \textit{discrete condensed anima}. On our way to fully faithfulness of the discrete functor we will introduce the \textit{global sections functor} $\Gamma\colon \CondAni \rightarrow \Ani$, which provides a right adjoint to the discrete functor (cf.\ Def.\ \ref{global}, Prop.\ \ref{leftadjoint}). Discrete condensed anima are defined as follows.
%
\begin{defprop}{(cf.\ Lemma \ref{constinC}, Prop.\ \ref{constsheaf}, Def.\ \ref{discr})}
Every extremally disconnected profinite set is given by a cofiltered limit $X=\mathrm{lim}_{i \in I}X_i$ of finite discrete spaces (cf.\ Prop.\ \ref{profinlim}). And every $X_i$ can be viewed as anima by the fully faithful embedding $\Set \hookrightarrow \Ani(\Set)$. Under these assumptions, we have:\\
To every anima $A\in \Ani$, one can functorially assign a condensed anima $$\underline{A}^{\mathrm{disc}}\colon \extdis^{\mathrm{op}}\rightarrow \Ani$$ defined by the formula
\begin{align*}
X \mapsto \colim_{i \in I^{\mathrm{op}}}\Hom_{\Ani}(X_i,A),
\end{align*}
where the colimit is filtered. This condensed anima preserves cofiltered limits in $\extdis$ and is called \textit{discrete condensed anima} attachted to $A$.
\end{defprop}
\section*{Homotopy groups for animated condensed sets}
In order to define homotopy groups for condensed sets, we want to establish a notion of homotopy groups for the $\infty$-category of condensed anima, respectively animated condensed sets, $\CondAni \simeq \Ani(\Cond)$. Then, one can transfer this notion to condensed sets via the fully faithful embedding $\Cond \hookrightarrow \CondAni$.\\ The $\infty$-category $\Ani$ is the classical $\infty$-topos to be used for homotopy theory and every CW-space is represented as an anima via the functor $\CW \rightarrow \Ani$. One can define \textit{simplicial homotopy groups} for objects in $\Ani$ (cf.\ Def. \ref{simphomgr}) such that for every CW-space the topological homotopy groups coincide with the simplicial homotopy groups of its image in $\Ani$. Unfortunately, there is no hope for defining ordinary homotopy groups on $\CondAni$ which extend the construction on topological spaces in a natural way. Meaning, it does not exist a left adjoint to the fully faithfully embedding $\Ani \hookrightarrow \CondAni$ which allows us to identify every condensed anima with an object of $\Ani$ (cf.\ Section \ref{example}) and such that its restriction to $\CW$ sends a CW-space to its image under $\CW \rightarrow \Ani$. \\ But at least, introducing pro-$\infty$-categories (Ch.\@ \ref{procatbase}) allows us to define homotopy pro-groups. For this, we use the fact that, under certain conditions, a functor on an $\infty$-category can be extended to a functor on a pro-$\infty$-category which has a left adjoint (cf.\ Prop.\ \ref{proext}), even when the functor we started with does not. Chapter \ref{homotopy} unites our efforts in proving the claims of the subsequent theorem. The corresponding results can be found in Section \ref{groups}. The aim of Section \ref{example} is to determine the homotopy pro-groups of profinite sets. This example illustrates the fact that $\Ani \hookrightarrow \CondAni$ cannot admit a left adjoint (cf.\ Rem.\ \ref{noleftadj}).
\begin{theo}
The fully faithful embedding $\Ani \hookrightarrow \CondAni$ extends to a unique embedding $\Pro(\Ani)\hookrightarrow \CondAni$ preserving cofiltered limits. Here, $\Pro(\Ani)$ denotes the pro-$\infty$-category of pro-anima. This functor admits a left adjoint functor $\CondAni \rightarrow \Pro(\Ani)$. Moreover, extending the homotopy group functors on (pointed) anima
\begin{align*}
&\pi_0\colon \Ani \ \rightarrow \Set\\
&\pi_1\colon \Ani_{*} \rightarrow \Grp\\
&\pi_n \colon \Ani_{*} \rightarrow \AbGrp
\end{align*}
to unique cofiltered limit preserving functors on the corresponding pro-$\infty$-categories 
\begin{align*}
&\Pro(\pi_0)\colon \Pro(\Ani) \ \rightarrow \Pro(\Set)\\
&\Pro(\pi_1)\colon \Pro(\Ani)_{*} \rightarrow \Pro(\Grp)\\
&\Pro(\pi_n) \colon \Pro(\Ani)_{*} \rightarrow \Pro(\AbGrp)
\end{align*}
yields, by composition with the left adjoint $\CondAni \rightarrow \Pro(\Ani)$, a notion of homotopy pro-groups on condensed anima
\begin{align*}
&\widetilde{\pi}_0\colon \CondAni \ \rightarrow \Pro(\Set)\\
&\widetilde{\pi}_1\colon \CondAni_{*} \rightarrow \Pro(\Grp)\\
&\widetilde{\pi}_n \colon \CondAni_{*} \rightarrow \Pro(\AbGrp),
\end{align*}
whose restriction to CW-spaces coincides with the ordinary notion of homotopy groups on topological spaces.
\end{theo}
\section*{Connection to \'etale homotopy theory}
Finally, I want to give an idea\footnote{This idea was also suggested by Peter Scholze in his email from September 2020.} of how the results of this master thesis might be of interest for other theories: To a sufficiently nice scheme $X$, one can assign a pro-anima via the construction of the Artin-Mazur étale homotopy type, i.e. an object of the pro-$\infty$-category $\Pro(\Ani)$. This can be considered as a condensed anima. In classical theory, the \'etale fundamental group of such a scheme $X$ is a profinite group obtained similarly to the construction for condensed anima in the preceding theorem. In this way, we receive a relation to the étale homotopy theory.

\chapter*{Acknowledgement}\index{Danksagung}
First of all, I would like to express my gratitude to my supervisor Prof.\@ Dr.\@ Torsten Wedhorn for suggesting the topic of this master thesis and still leaving me the freedom to decide according to my own interests in which direction I would like to research. The regular meetings have been very helpful in keeping track of my own process and getting feedback on it. Most of all, however, I would like to thank him for encouraging me to trust in my mathematical abilities and not being afraid to repeat this reinforcement. \\
\newline
I would like to express special thanks to Prof.\@ Dr.\@ Peter Scholze who together with Dr.\@ Dustin Clausen developed the theory of condensed sets on which this master thesis is based on. In his quick reply to an email from me in September 2020, he outlined an idea to define some kind of homotopy groups for condensed sets. Pursuing this idea will be the central objective of the second half.\\
\newline
Furthermore, my grateful thanks are extended to my family and friends who believed in me. In particular, I would like to thank my parents for supporting me since always and also during my studies enabling me to do what I enjoy, and my brothers for always being there for me. A big thanks goes to my friends Anton Güthge, Haolin Jiang, Carsten Litzinger, Till Rampe, Lukas Saul and Esther Thomas for their support and thorough proofreading.\\
\newline
By writing this master thesis, I learned a lot, not only regarding mathematical aspects but also beyond that. I am profoundly grateful for this experience.
\begin{flushright}
Catrin Mair, February 2021
\end{flushright}
%
%
\chapter{Topological and categorical preliminaries}
I expect the reader to be acquainted with fundamental concepts of category theory and topology, especially higher categories and homotopy theory. This chapter shall give an overview and a summary of the most important basics and common results beyond these topics. These will be necessary to follow the proofs appearing in this master thesis. Most of the time, the order of this chapter will represent the occurrence of results in the following chapters. Underlying sources, which do not occur as explicit citation, as well as suggestions for further reading are specified at the beginning of the sections in this chapter and the following chapters.
\section{Topological notions}
The category of compact Hausdorff spaces and continuous maps $\CoHa$ as well as its subcategories of totally disconnected spaces $\mathcal{S}\mathrm{tone}$ and extremally disconnected spaces $\extdis$ play an important role in the definition of \textit{condensed sets}. This section aims to provide an overview of these notions and to collect all basic topological results we will need in order to properly work with the definition of condensed sets. The category $\mathcal{S}\mathrm{tone}$ is also known as the category $\PFin$ of profinite sets. This will be discussed in Section \ref{procat}.
\subsection{Compact Hausdorff spaces}
\begin{proposition}\label{surjcomphaus}
Every surjection of compact Hausdorff spaces is a quotient map.
\end{proposition}
\begin{proof}
A well known result states that maps from compact spaces to Hausdorff spaces are closed. Indeed, every closed set in a compact space is compact and its image under a continuous map is compact again. As compact sets in Hausdorff spaces are closed, the image is closed. The statement follows as surjective closed maps are automatically quotient maps.
\end{proof}
\begin{definition and proposition}\textbf{Stone-\v{C}ech compactification}\label{stonecech}\\
To every topological space $X$, one can assign  a compact Hausdorff space $\beta X$ and a continuous map $\iota_X \colon X \rightarrow \beta X$ satisfying the following universal property: For every compact Hausdorff space $K$ and every continuous map $f\colon X\rightarrow K$ there exists a unique continuous map $\beta f\colon \beta X\to K$ s.t. $ f=\beta f\circ \iota _{X}$. The space $\beta X$ is called \textit{Stone-\v{C}ech compactification} of $X$. The assignment $X\mapsto \beta X$ yields a left adjoint functor to the embedding functor from the category of compact Hausdorff spaces to the category of topological spaces.
\end{definition and proposition}
\begin{proof}
A construction of the compactification and a proof of the universal property  can be found in \cite[\href{https://stacks.math.columbia.edu/tag/0908 }{Tag 0908}]{Stacks} and \cite[\href{https://stacks.math.columbia.edu/tag/0909}{Tag 0909}]{Stacks}. The left adjointness follows as a direct corollary from the universal property.
\end{proof}
\begin{corollary}\label{limitcomp}
Every limit of compact Hausdorff spaces is compact Hausdorff.
\end{corollary}
\begin{proof}
This directly follows from the left adjointness in the previous proposition.
\end{proof}
Recall that a retract of a topological space $X$ is a topological subspace $A$ that is fixed under some retraction $r$, i.e. a continuous map $r \colon A\rightarrow X$ such that its restriction to $A$ is the identity.
\begin{lemma}\label{compstabprop}
Compact Hausdorff spaces are closed under finite coproducts and retracts.
\end{lemma}
\begin{proof}
A covering of a finite coproduct of compact Hausdorff spaces is open if and only if its restriction to all spaces of the coproduct is open. As these are compact by assumption, there exist finite subcoverings of the restrictions. Then, the finite union of these subcoverings forms a finite subcovering of the coproduct. This shows compactness for every finite coproduct.\\ Two points in a coproduct are automatically separated if they are contained in different spaces. If they are in the same space, disjoint open neighbourhoods are given by assumption of Hausdorffness for all spaces forming the coproduct. Hence, all coproducts of Hausdorff spaces are Hausdorff.
Now, consider a retract of a topological space $X$.
Then, Hausdorffness easily follows as subspaces of Hausdorff spaces are Hausdorff. For compactness, we can use that retracts of Hausdorff spaces are closed, and closed subsets of compact spaces are compact. Precisely, the retract is the preimage of the diagonal in $X \times X$ under the continuous map $x \mapsto (x, r(x))$. The diagonal is closed by Hausdorffness of $X$. Thus, the retract is also closed.
\end{proof}
\subsection{Totally and extremally disconnected spaces}
\begin{definition}{\textbf{Totally and extremally disconnected}}\label{extdisdef}
\begin{itemize}
\item[1.] A topological space is called \textit{totally disconnected} if all connected subspaces are singletons.
\item[2.] A topological space is called \textit{extremally disconnected} if the closure of every open subset is open. 
\end{itemize}
\end{definition}
\begin{remark}
Hausdorff extremally disconnected spaces are totally disconnected. Indeed for any two points in a connected subspace, one can find  disjoint open neighbourhoods. Then, the closure of such an open neighbourhood still does not contain the other point and is a proper open and closed subset, in contradiction to the connectedness. Hence, the connected subspace cannot contain more than one point.
\end{remark}
If we assume in addition that the considered topological space is compact Hausdorff, there are equivalent characterizations of extremally disconnected:
\begin{proposition}\label{characextdis} Let $X$ be a compact Hausdorff space. Then the following are equivalent.
\begin{itemize}
\item[(i)] The space $X$ is extremally disconnected.
\item[(ii)] For any surjective continuous map $f\colon Y \rightarrow X$ with Y compact Hausdorff, there exists a continuous section.
\item[(iii)] For any map $X \rightarrow Z$ and surjective map $Y \rightarrow Z$ of compact Hausdorff spaces, there exists a continuous map $X \rightarrow Y$ s.t. the diagram
\[
  \begin{tikzcd}
   X \arrow[dotted]{r} \arrow[swap]{dr} & Y \arrow{d} \\
     & Z
  \end{tikzcd}
\]
commutes.
\end{itemize}
\end{proposition}
\begin{proof}
See \cite[\href{https://stacks.math.columbia.edu/tag/08YH}{Tag 08YH}]{Stacks}.
\end{proof}
\begin{definition}{\textbf{Disconnected compact Hausdorff spaces}}\\
Let $\CoHa$ denote the category of all compact Hausdorff spaces with continuous maps. We define two full subcategories
$$\extdis\subset \mathcal{S}\mathrm{tone}\subset \CoHa, $$ where
\begin{itemize}
\item[1.] the category $\mathcal{S}\mathrm{tone}$ is the subcategory of totally disconnected compact Hausdorff spaces. Its objects are called \textit{Stone spaces}.
\item[2.] the category $\extdis$ is the subcategory of extremally disconnected compact Hausdorff spaces, i.e. of spaces satisfying the equivalent properties of Proposition \ref{characextdis}. Some authors refer to its objects as \textit{Stonean spaces} (cf.\ \cite{Barwick2019}).
\end{itemize} 
\end{definition}
\begin{corollary}\label{stonecechext}
The Stone-\v{C}ech compactification of a discrete space is an extremally disconnected space. 
\end{corollary}
\begin{proof}
Let $X$ be a discrete space and $\beta X$ its Stone-\v{C}ech compactification. As $\beta X$ is compact Hausdorff, we can apply Proposition \ref{characextdis} as follows: 
If a map $h_1 \colon \beta X \rightarrow Z$ and a surjective map $h \colon Y \rightarrow Z$ of compact Hausdorff spaces are given, there exist continuous maps $f \colon X \overset{\iota_X}{\rightarrow} \beta X \overset{h_1}{\rightarrow} Z$ and $g \colon X \rightarrow Y$ s.t. 
\[
  \begin{tikzcd}
   X \arrow[r, dotted, "g"] \arrow[swap]{dr}{f} & Y \arrow{d}{h} \\
     & Z
  \end{tikzcd}
\]
commutes. Here, $g$ sends an element $x \in X$ to an arbitrary preimage in $h^{-1}(f(x))$. The map is continuous as $X$ is discrete. By the universal property of $\beta X$ we can lift $g$ to a map $\beta g \colon \beta X \rightarrow Y$ with $g=\beta g \circ \iota_X$. As one has $$f=h\circ g=h\circ \beta g \circ \iota_X,$$ the composition $h \circ \beta g$ is by uniqueness in Proposition \ref{stonecech} equal to $h_1$ and Condition (iii) of Proposition \ref{characextdis} is fulfilled.
\end{proof}
\begin{corollary}\label{Surjext}
Every compact Hausdorff space admits a surjection from an extremally disconnected space.
\end{corollary}
\begin{proof}
For any compact Hausdorff space $X$, we have a surjection $\widetilde{X} \rightarrow X$, where $\widetilde{X}$ is the underlying set of $X$ equipped with the discrete topology. We can lift this map to a map $\beta \widetilde{X} \rightarrow X$ by Proposition \ref{stonecech} with $\beta \widetilde{X}$ extremally disconnected by Corollary \ref{stonecechext}. 
\end{proof}
\begin{theorem}\label{coeqext}
Every compact Hausdorff space is a coequalizer of extremally disconnected compact Hausdorff spaces in the category of all topological spaces.
\end{theorem}
\begin{proof}
Let $K$ be a compact Hausdorff space. Due to Corollary \ref{Surjext}, there exists a surjection $S\rightarrow K$ from an extremally disconnected compact Hausdorff space $S$. Moreover, as $S\times_K S$ is compact Hausdorff by Corollary \ref{limitcomp}, it also admits a surjection from some extremally disconnected compact Hausdorff space $S^{\prime}$. We claim that the diagram 
\[
\begin{tikzcd}
S^{\prime}  \ar[r,shift left=.75ex, "\tilde{p}_1"]
  \ar[r,shift right=.75ex,swap, "\tilde{p}_2"]
&
S \ar[r,"q"] 
&
K,
\end{tikzcd}
\]
where the maps $\tilde{p}_1, \tilde{p}_2$ are defined as compositions by
\[
\begin{tikzcd}
S^{\prime} \ar[r, "f"]
&
S\times_K S \ar[r,shift left=.75ex,"p_1"]
  \ar[r,shift right=.75ex,swap,"p_2"]
&
S,
\end{tikzcd}
\]
is a coequalizer. 
By definition of the fibre product $S\times_K S$, one clearly has $$q\circ (p_1\circ f)=q \circ (p_2 \circ f).$$ Consider some topological space $T$ and a continuous map $h\colon S \rightarrow T$ s.t. $$h\circ (p_1\circ f)=h \circ (p_2 \circ f)$$ or, equivalently 
\begin{align}\label{comm}
h\circ p_1=h \circ p_2
\end{align} as $f$ is surjective. Since any surjection of compact Hausdorff spaces is a quotient map by Proposition \ref{surjcomphaus}, any map $K\rightarrow T$ for which the composite $S\rightarrow K \rightarrow T$ is continuous is continuous itself.
Thus, it is enough to define a map on the underlying sets $K\rightarrow T$ s.t. the diagram
\begin{equation}\label{diagr}
  \begin{tikzcd}
  S \arrow{r}{q} \arrow[swap]{dr}{h} & K \arrow{d} \\
     & T
  \end{tikzcd}
\end{equation}
commutes. \\
For commutativity, we have to define the image of some $a \in K$ under the map $K\rightarrow T$ as image of some $\tilde{a} \in q^{-1}(a)$ under $h$. To receive a well-defined mapping, we have to assure that all elements of $q^{-1}(a)$ are mapped on the same element under $h$. This is clearly given, as by definition of $S\times_K S$ for all elements $x,y \in S$ with $q(x)=q(y)$ one has $(x,y)\in S\times_K S$ and by (\ref{comm}) then $h(x)=h(y)$. Thus, we can define a map $K\rightarrow T$ satisfying (\ref{diagr}), which is clearly unique by surjectivity of $S\rightarrow K$.
\end{proof}
\section{Categorical notions}
I want to start this section with a short remark about set theoretic conventions we will make for simplicity. The following subsection is based on \cite[Section 2]{Day2007}.
\subsection{Set theoretic conventions}\label{warn}
Many fundamental constructions in category theory demand the appearing categories to be small. An important example for this is the category of presheaves $\Fun(\cat^{\mathrm{op}}, \Set)$ on a small category $\cat$. For a large category, this construction may possibly not even result in a category as its hom-sets are not small in general. To solve this issue one can only consider  \textit{small presheaves} defined as in the following. 
\begin{definition}{\textbf{Small functors}}\\
Let $\cat$ be a possibly large category. A functor $F\colon \cat \rightarrow \Set$ is said to be \textit{small} if it is the left Kan extension of some presheaf with a small domain. Equivalently, it is the left Kan extension of its restriction to some small full subcategory of $\cat$.
\end{definition}
\begin{lemma}\label{smallcolimits}
The small functors are precisely given by the small colimits of representables. 
\end{lemma}
\begin{proof}
See \cite[paragraph before Remark 2.1.]{Day2007}.
\end{proof}
\begin{remark}
The totality of small presheaves on a fixed category $\cat$ forms an ordinary category which is cocomplete, i.e. it contains all small colimits, as small colimits of small colimits are still small. It can be viewed as the full subcategory of the presheaf category spanned by small functors and is in fact
the free cocompletion of $\cat$ under the Yoneda embedding by small colimits. 
\end{remark}
\begin{warning}{\textbf{Dealing with large categories}}\\
Throughout this master thesis, we will consider large categories and want, amongst other things, to work with their presheaf categories. In order to handle set theoretic problems arising from this consideration in a simple way, we will implicitly assume functors to be small. For our purposes, it is enough to know that there is a way to deal with these problems. Nevertheless, we will mention the existence of set theoretic problems at relevant points and refer to some possibility of solving this issue.
\end{warning}
\subsection{Pro-categories in the ordinary setting}\label{procat}
In this section we provide some general background on pro-categories. The higher categorical analogue will be discussed in Chapter \ref{procatbase}. The basic definitions and results on pro-categories provided in this section are taken from \cite[Chapter 6]{Johnstone1982}. 
\begin{definition}{\textbf{Filtered category}}\\
A non-empty index category $I$ is called \textit{filtered} if the following conditions are satisfied. 
\begin{itemize}
\item[1.]For every pair of objects $x,y$ of $I$, there exists an object $z$ and morphisms $x\rightarrow z$, $y \rightarrow z$, and
\item[2.]for every pair of objects $x,y$ and every pair of morphisms $f,g\colon x \rightarrow y$ of $I$, there exists a morphism $h\colon y \rightarrow z$ of $I$ such that $h\circ f=h\circ g$ as morphisms in $I$.
\end{itemize}
A category $I$ is called \textit{cofiltered} if its opposite category $I^{\mathrm{op}}$ obtained by reversing all arrows in $I$  is filtered.
\end{definition}
\begin{definition}{\textbf{Filtered and cofiltered diagrams}}\\
A diagram with filtered (resp. cofiltered) index category is called \textit{filtered} (resp. \textit{cofiltered}). Accordingly, filtered (co)limits (resp. cofiltered (co)limits) are (co)limits with filtered (resp. cofiltered) index categories.
\end{definition}
\begin{definition}{\textbf{Pro-categories}}\\ 
Let $\cat$ be a (small) category. We define Pro($\cat$) to be the category whose objects are diagrams $X \colon I\rightarrow \cat$ over a small and cofiltered category $I$ and whose morphism sets are given by
$$\Hom_{\mathrm{Pro}(\cat)}(X, Y ) \coloneqq \mathrm{lim}_{j \in J} \colim_{i \in I} \Hom_{\cat}(X_i, Y_j),$$
where $X_i = X(i)$ and $Y_j=Y(j)$.
Composition of morphisms can be defined in a natural way. We refer to Pro($\cat$) as the \textit{pro-category} of $\cat$ and to objects of Pro($\cat$) as \textit{pro-objects}.
\end{definition}
For simplicity, a pro-object $X \colon I \rightarrow \cat$ will often be written as a cofiltered system $X = \{X_i\}_{i\in I}$.
\begin{remark}
A pro-category can also be defined as the dual notion of an ind-category. More precisely, the category of pro-objects in $\cat$ is the opposite category of that of ind-objects in the opposite category of $\cat$, i.e. one has
$$\mathrm{Pro}(\cat)\cong(\mathrm{Ind}(\cat^{\mathrm{op}}))^\mathrm{op}.$$
\end{remark}
\begin{proposition}{\textbf{Properties of pro-categories}}\\
Let $\cat$ be a (small) category.
\begin{itemize}
\item[1.] There is a canonical full embedding  $i \colon \cat \hookrightarrow \mathrm{Pro}(\cat)$ which associates to $X \in \cat$ the diagram $\widetilde{X}\colon 1 \rightarrow \cat$ which sends the unique object of the terminal index category $1$ to $X$.
\item[2.] The category $\mathrm{Pro}(\cat)$ is equivalent to the full subcategory of $\Fun(\cat, \Set)^{\mathrm{op}}$ whose objects are cofiltered limits of $\{\Hom_{\cat}(X_i, \cdot)\}_{i \in I}$ with $X_i \in \cat$ for all $i \in I$. It admits all cofiltered limits.
\end{itemize}
\end{proposition}
\begin{proof}
The statements are formulated for ind-categories in Section 1.2. and Section 1.4. of \cite[Chapter VI]{Johnstone1982}. With regards to Section 1.9 in the same chapter, they can be transferred to the dual notion of pro-categories.
\end{proof}
\begin{remark}
Regarding part (2) of the proposition, we will speak of pro-objects as small cofiltered limits of objects in $\cat$. A pro-object in the essential image of the embedding $\cat \hookrightarrow \mathrm{Pro}(\cat)$ will be called \textit{corepresentable}. 
\end{remark}
The pro-category of profinite sets defined below will play an important role for us. The subsequent proposition states that this category actually already appeared in the first section of this chapter. We will specify how it can be described in terms of limits of certain topological spaces.
\begin{definition}{\textbf{Profinite sets}}\label{profindef}\\
A \textit{profinite set} is a pro-object in the category $\Set^{\mathrm{fin}}$ of finite sets. We denote the category of profinite sets by $\PFin$.
\end{definition}
\begin{proposition}\label{profinlim}
The category $\PFin$ of profinite sets is equivalent to the category $\mathcal{S}\mathrm{tone}$ of \textit{Stone spaces}, i.e. totally disconnected compact Hausdorff spaces. The equivalence is given by considering a profinite set as a cofiltered diagram of finite discrete spaces and taking the limit in the category of topological spaces. Using this identification of $\PFin$ and $\mathcal{S}\mathrm{tone}$, the fully faithful embedding of finite sets into the category of profinite sets is given by equipping them with the discrete topology.
 \end{proposition}
\begin{proof}
For a full proof of the equivalence of profinite sets and Stone spaces using Stone duality see \cite[Chapter VI 2.3.]{Johnstone1982}. One can even construct the equivalence explicitly as described in the proposition. This yields an essentially surjective functor since every Stone space can be represented as a cofiltered limit of finite discrete spaces \cite[\href{https://stacks.math.columbia.edu/tag/08ZY}{Tag 08ZY}]{Stacks}. Then, the fully faithful embedding can clearly be obtained by equipping finite sets with the discrete topology.
\end{proof}
\begin{remark}
In the proof of the previous proposition, we noted that to every Stone space $X$ one can attach a cofiltered projective system of finite sets. We want to review how this can be realized:\\ If $X=\amalg_{i \in I}U_i$ is a finite disjoint union of non-empty open and closed subsets, every $U_i$, that is not a singleton, can be written itself as finite disjoint union of proper open and closed subsets since we assume $X$ to be totally disconnected. Hence, every such covering $\{U_i\}_{i \in I}$ can be refined by a covering $\{V_j\}_{j \in J}$, i.e. for every $i \in I$ there exists a $J_i\subset J$ with $U_i=\amalg_{j \in J_i}V_j$. Particularly, one can find a common refinement for two given decompositions. Let $\mathcal{I}$ be the set
of all finite decompositions of $X$ by open and closed subsets. For every $I \in \mathcal{I}$, we define a quotient space $X_I$ by setting  elements $x, y \in X$ to be equivalent if they are contained in the same $U_i$ of the decomposition. Then $X_I$ defines a finite discrete space. We have a canonical map $X_{I'}\rightarrow X_I$ of finite discrete spaces if the covering corresponding to $I'$ refines the covering corresponding to $I$. By this, we can view $\mathcal{I}$ as a cofiltered category with morphisms $I'\rightarrow I$ corresponding to refinements as for any two decompositions one can find a common refinement. Together with the maps $X\rightarrow X_I$, we obtain a cofiltered projective system and thus a continuous  map
$X\rightarrow \mathrm{lim}_{I \in \mathcal{I}}X_I,$
which is indeed a homeomorphism as argued in \cite[\href{https://stacks.math.columbia.edu/tag/08ZY}{Tag 08ZY}]{Stacks}. 
\end{remark}
\subsection{Grothendieck topologies and sheaves}
In this section, we will review the theory of sheaves associated to Grothendieck topologies. There is a very common definition of \textit{Grothendieck topologies} which only applies for categories with certain fibre products. In the following, we will call them \textit{Grothendieck pretopologies}. 
We will shortly discuss the original definition of a \textit{Grothendieck topology}, which does not require the existence of fibre products, since this will enable us to define \textit{condensed sets} in the next chapter in a more proper way. Note Remark \ref{pretoandtopo} for a comment about how the definitions mentioned above are related. This section is mainly based on \cite[Appendix B]{Lurie2018}.
\begin{definition}{\textbf{Grothendieck pretopology}}\label{cov1}\\
Let $\cat$ be a category with fibre products. A \textit{Grothendieck pretopology} on $\cat$ is a  collection of families $\{U_i \rightarrow U\}_{i \in I}$ of morphisms for each object $U\in \cat$, called \textit{covering families of $U$} or just \textit{coverings}, such that:
\begin{itemize}
\item[1.]Every isomorphism $V\xrightarrow{\cong} U$ forms a covering family $\{V\xrightarrow{\cong} U\}$.
\item[2.]If $\{U_i \rightarrow U\}_{i \in I}$ is a covering family and $f\colon V \rightarrow U$ is any morphism in $\cat$, then $\{V\times_U U_i \rightarrow V \}_{i \in I}$ is a covering family.
\item[3.]If $\{U_i\rightarrow U\}_{i\in I}$ is a covering family and for each $i$ also $\{U_{i,j}\rightarrow U_i\}_{j\in J_i}$ is a covering family, then the family of composites $\{U_{i,j}\rightarrow U_i \rightarrow U\}_{i \in I, j\in J_i}$ is a covering family as well.
\end{itemize}
We call the pair of such a category together with a family of coverings a $\textit{site}$.
\end{definition}
\begin{example}{\textbf{Pretopology on $\CoHa$ and $\mathcal{S}\mathrm{tone}$}}\label{examplepretopo}\\
Consider the category $\cat\coloneqq\CoHa$ (respectively $\cat\coloneqq\mathcal{S}\mathrm{tone}$) of (totally disconnected) compact Hausdorff spaces which is closed under limits.  Let $X$ be a (totally disconnected) compact Hausdorff space and $\{X_i\rightarrow X\}_{i\in I}$ a finite family of jointly surjective maps in $\cat$, i.e. the induced map $\amalg_{i \in I} X_i \rightarrow X$ is surjective. These families define a Grothendieck pretopology on $\cat$. Note that the category $\extdis$ of extremally disconnected compact Hausdorff spaces does not have fibre products. Hence, for this category we cannot define a Grothendieck pretopology in the same manner as in Definition \ref{cov1}. This issue will be resolved by Definition \ref{Grotop}.
\end{example}
The definition of a Grothendieck pretopology enables us to define a notion of sheaves for categories with certain fibre products that provides a nice extension of the notion of sheaves for topological spaces (defined on the category of open subsets of a topological space).
\begin{definition}{\textbf{Sheaves on sites}}\label{shv1}\\
Let $\cat$ be a site and let $F$ be a presheaf of sets on $\cat$, i.e. a contravariant functor $F\colon \cat^{\mathrm{op}}\rightarrow \Set$. We say that $F$
is a sheaf on the site $\cat$ if for all $U \in \cat$ and every covering $\{U_i \rightarrow U\}_{i\in I}$ the diagram
\[
\begin{tikzcd}[column sep = huge]
F(U) \ar[r]
&
\prod_{i \in I} F(U_i) \ar[r,shift left=.75ex, "\prod_{ij}F(p_{ij}^1)"]
  \ar[r,shift right=.75ex,swap, "\prod_{ij}F(p_{ij}^2)"]
&
\prod_{i,j}F(U_{i}\times_U U_{j})
\end{tikzcd}
\]
is an equalizer, where 
\begin{align*}
&p_{ij}^1\colon U_{i}\times_U U_{j} \rightarrow U_i\\
&p_{ij}^2\colon U_{i}\times_U U_{j} \rightarrow U_j
\end{align*} are the canonical projections.
The category of sheaves on $\cat$ is denoted by $\mathrm{Sh}(\cat)$.
\end{definition}
The following definition lays the foundation for a more general notion of a topology and hence for sheaves on a category that does not admit certain fibre products, e.g. the category of extremally disconnected compact Hausdorff spaces. In a nutshell, a covering given by a collection of families of morphisms will be replaced by one given by a collection of \textit{sieves}. And the covering family $\{V\times_U U_i\rightarrow V\}_{i \in I}$ of Definition \ref{cov1} (2) will be replaced by a so called \textit{pullback of a sieve} along some morphism $f$.
\begin{definition}{\textbf{Sieves on objects}}\\
Let $\cat$ be a category and $X\in \cat$. Recall that the over category $\cat_{/X}$ over $X$ consists of all pairs $(Y, f\colon Y \rightarrow X)$ whereby $f$ is a morphism in $\cat$. Morphisms $(Y^{\prime},f^{\prime})\rightarrow (Y,f)$ in $\cat_{/X}$ are defined by morphisms $g\colon Y^{\prime}\rightarrow Y$ in $\cat$ such that $f\circ g=f^{\prime}\colon Y^{\prime} \rightarrow X$.
\begin{itemize}
\item[1.] A \textit{sieve} on $X$ is a full subcategory $S(X)$ of the over category $\cat_{/X}$ consisting of pairs $(Y,f\colon Y \rightarrow X)$  such that for every morphism $g\colon (Y^{\prime}, f^{\prime}) \rightarrow (Y,f)$ in $\cat_{/X}$ the pair $(Y^{\prime}, f^{\prime} \colon Y^{\prime}\rightarrow X)$ belongs to $S(X)$ whenever $(Y, f\colon Y \rightarrow X)$ does.
\item[2.] The \textit{pullback} of a sieve $S(X)$ on $X$ along some morphism $f\colon Y\rightarrow X$ is defined as the sieve on $Y$ consisting of those $(V,g\colon V \rightarrow Y)$ for which the composite $(V, f\circ g \colon V \rightarrow X)$ belongs to $S(X)$. It is denoted by $f^{*}S(X)$.
\end{itemize}
\end{definition}
\begin{definition}{\textbf{Sieves generated by morphisms}}\\
Let $\cat$ be a category, $X\in \cat$ and $\{f_i\colon X_i \rightarrow X\}_{i \in I}$ a collection of morphisms with target $X$. The smallest sieve on $X$ which contains every $(X_i, f_i)$ is called \textit{sieve generated by the morphisms} $f_i$. It is given as the full subcategory of $\cat_{/X}$ spanned by those $(Y, f\colon Y \rightarrow X)$ for which there exists a factorization $Y \rightarrow X_i \rightarrow X$ for some $i \in I$.
\end{definition}
\begin{definition}{\textbf{Grothendieck topology}}\label{Grotop}\\
Let $\cat$ be an arbitrary category. A \textit{Grothendieck topology} $J$ on $\cat$ is a collection of sieves for each object $X\in \cat$, called \textit{covering sieves} and denoted by $J(X)$, such that:
\begin{itemize}
\item[1.]For each object $X\in \cat$, the category $\cat_{/X}$ (called maximal sieve) is a covering sieve on $X$.
\item[2.]For each morphism $f\colon Y\rightarrow X$ in $\cat$ and each covering sieve $S(X) \in J(X)$ on $X$, the pullback $f^{*}S(X)$ is a covering sieve on $Y$.
\item[3.] Let $S(X)\in J(X)$ be a covering sieve on $X$ and let $T(X)$ be another sieve on $X$ such that for each morphism $f\colon Y \rightarrow X$ belonging to $S(X)$ the pullback $f^{*}T(X)$ is a covering sieve on $Y$. Then $T(X)$ is a covering sieve on $X$.
\end{itemize}
A category together with a Grothendieck topology is also called a $\textit{site}$.
\end{definition}
\begin{example}{\textbf{Coherent topology}}\label{exampletopo}\\
Recall that an effective epimorphism in a category is a morphism $f\colon X\rightarrow Y$ such that $X\times_Y X$ is defined and 
\[
\begin{tikzcd}
X\times_Y X \ar[r,shift left=.75ex]
  \ar[r,shift right=.75ex,swap]
&
X \ar[r, "f"]
&
Y
\end{tikzcd}
\]
is a  coequalizer.
Let $\cat$ be a coherent category (e.g. a regular and extensive category or a (pre)topos). A sieve on some object $X\in \cat$ is a \textit{coherent covering sieve} if it contains a finite collection of morphisms $\{X_i \rightarrow X\}_{i \in I}$ for which the induced map $\amalg_{i \in I} X_i \rightarrow X$ is an effective epimorphism. The collection of coherent covering sieves determines a Grothendieck topology on $\cat$ by \cite[Proposition B.5.2.]{Lurie2018}, called the \textit{coherent topology} on $\cat$. 
\end{example}
\begin{definition}{\textbf{Coverings}}\label{cov2}\\
Let $\cat$ be equipped with a Grothendieck topology. A collection $\{X_i \rightarrow X\}_{i\in I}$ of morphisms in $\cat$ is called a \textit{covering} if it generates a covering sieve on $X \in \cat$.
\end{definition}
One may note that there seems to be an abuse of notation by defining \textit{covering} in Definition \ref{cov1} and Definition \ref{cov2}. The following remark dissolves this vagueness. 
\begin{remark}\label{pretoandtopo}
For categories with fibre products, every \textit{Grothendieck pretopology} can be converted into a \textit{Grothendieck topology}.  
Namely, by \cite[\href{https://stacks.math.columbia.edu/tag/00ZC}{Tag 00ZC}]{Stacks}, the collection of all sieves that contain a covering family from a chosen pretopology forms a Grothendieck topology. Thus, every site $\cat$ determines a topology called \textit{topology associated to} $\cat$ \cite[\href{https://stacks.math.columbia.edu/tag/00ZD}{Tag 0ZD}]{Stacks}
and every covering family satisfying the axioms of Definition \ref{cov1} also is  a covering in the sense of Definition \ref{cov2}. 
\end{remark}
\begin{example}{\textbf{Coherent topology on $\CoHa$ and $\mathcal{S}\mathrm{tone}$}}\label{cohcompH}\\
The category $\CoHa$ of compact Hausdorff spaces is a pretopos and the category $\PFin=\mathcal{S}\mathrm{tone}$ of totally disconnected compact Hausdorff spaces is regular extensive as this is true for $\SFin$ \cite[Warning 6.1.23.]{Lurie2018}. Therefore, both categories are coherent. One can define the coherent topology as in Example \ref{exampletopo} for them. Moreover, in these cases, every surjective continuous map is an effective epimorphism as one can show, similarly to the argument in the proof of Proposition \ref{coeqext}, that a diagram
\[
\begin{tikzcd}
S\times_K S \ar[r,shift left=.75ex,"p_1"]
  \ar[r,shift right=.75ex,swap,"p_2"]
&
S \ar[r,"q"] 
&
K 
\end{tikzcd}
\]
with $q$ surjective is a coequalizer. Conversely, every epimorphism of compact Hausdorff spaces is surjective. Thus, the coherent topology is exactly the topology associated to the pretopology of Example \ref{examplepretopo}.
\end{example}
\begin{definition}{\textbf{Sheaves on sites}}\label{shv2}\\
Let $\cat$ be a category equipped with a Grothendieck topology.
We say that a functor $F\colon \cat^{\mathrm{op}}\rightarrow \Set$ is a \textit{sheaf} if for each object $X\in \cat$ and each covering sieve $J(X)$ the canonical map
$$F(X)\rightarrow \lim_{D \in J(X)^{\mathrm{op}}} F(D)$$ is a bijection. We let $\mathrm{Sh}(\cat)$ denote the full subcategory of $\Fun(\cat^{\mathrm{op}}, \Set)$ spanned by those functors which are sheaves on $\cat$.
\end{definition}
The following remark assures that our two definitions of sheaves are compatible for categories with fibre products.
\begin{remark}
For categories with fibre products, a functor is a sheaf on a pretopology if and only if it is a sheaf on the corresponding Grothendieck topology of Remark \ref{pretoandtopo} by \cite[\href{https://stacks.math.columbia.edu/tag/00ZC}{Tag 00ZC}]{Stacks}. 
\end{remark}
\begin{example}{\textbf{Sheaves on coherent topology}}\label{sheavcoh}  \\
Let $\cat$ be a category equipped with the coherent topology as defined in Example \ref{exampletopo}. Then by \cite[Proposition B.5.5.]{Lurie2018}, a contravariant functor $F$ is a sheaf if and only if the following conditions are satisfied:
\begin{itemize}
\item[(I)] The functor preserves finite products, i.e. for every finite collection of objects $\{X_i\}_{i \in I}$ of $\cat$, the canonical map 
\begin{align*}F(\amalg_{i \in I}X_i) \rightarrow \prod_{i \in I}F (X_i)
\end{align*}
is bijective.
\item[(II)]For every effective epimorphism $Y\rightarrow X$ in $\cat$, the diagram of sets
\[
\begin{tikzcd}
F(X) \ar[r]
&
F(Y) \ar[r,shift left=.75ex]
  \ar[r,shift right=.75ex,swap]
&
F(Y\times_X Y)
\end{tikzcd}
\]
is an equalizer.
\end{itemize}
\end{example}
The following will be useful in order to work with condensed sets, which will be introduced in the next chapter. See Proposition \ref{condsetequi} for an application of the results stated below.
\begin{definition}{\textbf{Basis}}\label{basis}\\
Let $\cat$ be a category equipped with a Grothendieck topology. We call a full
subcategory $\mathcal{D} \subset \cat$ a \textit{basis} for $\cat$ if for every object $X\in \cat$, there exists a covering $\{f_i\colon D_i \rightarrow X\}_{i \in I}$, where
the set $I$ is small and each $D_i$ belongs to $\mathcal{D}$.
\end{definition}
\begin{proposition}\label{basistopo} Let $\cat$ be a category equipped with a Grothendieck topology and let $\mathcal{D} \subset \cat$ be a basis. Then, there is a unique Grothendieck topology on the category $\mathcal{D}$ such that a collection of morphisms $\{D_i \rightarrow D\}_{i \in I}$ in $\mathcal{D}$ is a covering if and only if it is a covering in $\cat$. Moreover, precomposition with the embedding $\mathcal{D} \hookrightarrow \cat$ induces an equivalence of categories 
$$\mathrm{Shv}(\cat) \rightarrow \mathrm{Shv}(\mathcal{D}), F \mapsto F|_{\mathcal{D}^{\mathrm{op}}}.$$
\end{proposition}
\begin{proof}
See \cite[Proposition B.6.3]{Lurie2018}.
\end{proof}
More precisely, the second statement is a consequence of the following.
\begin{proposition}\label{rightKan}
Let $\cat$ be a category equipped with a Grothendieck topology, let $\mathcal{D}\subset \cat$ be a basis equipped with the Grothendieck topology from Proposition
\ref{basistopo}, and let
$F\colon \cat^{\mathrm{op}}\rightarrow \Set$ be a functor. Then, $F$ is a sheaf if and only if it satisfies the following pair of conditions:
\begin{itemize}
\item[1.] The restriction $F|_{\mathcal{D}^{\mathrm{op}}}\colon \mathcal{D}^{\mathrm{op}}\rightarrow \Set$ is a sheaf.
\item[2.] The functor $F$ is a right Kan extension of its restriction $F| _{\mathcal{D}^{\mathrm{op}}}$. 
\end{itemize}
\end{proposition}
\begin{proof}
See \cite[Proposition B.6.4]{Lurie2018}.
\end{proof}
One can apply the topos-theoretic concepts of quasicompactness and quasiseparatedness, defined as in \cite{Haine2019}, to sheaves on any (coherent) site. 
\begin{definition}{\textbf{Quasicompact and quasiseparated sheaves}}\label{quasi}\\
Let $\cat$ be a (small) category with Grothendieck topology and $\mathrm{Sh}(\cat)$ the corresponding category of sheaves. 
\begin{itemize}
\item[1.] A sheaf $F \in \mathrm{Sh}(\cat)$ is called \textit{quasicompact} if every covering has a finite subcovering. This means, for every covering $\{U_i \rightarrow F\}_{i \in I}$, there exists a finite subset $J\subset I$ for which the collection of maps $\{U_i \rightarrow F\}_ {i \in J}$ is also a covering. 
\item[2.] A sheaf $F \in \mathrm{Sh}(\cat)$ is called \textit{quasiseparated} if for every pair of morphisms $U \rightarrow F \leftarrow V$, where $U$ and $V$ are quasicompact, the fibre product $U \times_F V$ is also quasicompact. 
\end{itemize}
\end{definition}
\subsection{The formalism of \texorpdfstring{$\infty$}{infinity}-categories}\label{forminf}
There are different definitions realizing the general idea of an $\infty$-category. We will define $\infty$-categories by quasi-categories as in \cite{Kero}. This definition arises directly from the idea behind $\infty$-categories: We want a class of mathematical objects which can behave like categories and like topological spaces up to homotopy. The formalism of $\infty$-categories we formulate in this section is based on \cite{Kero, Hoermann2019,Lurie2009}.
\begin{reminder}{\textbf{$\infty$-categories by quasi-categories}}\\
Ordinary categories as well as topological spaces can be viewed as simplicial sets by the construction of the nerve of a category resp. the singular simplicial complex of a topological space. Both constructions represent classes of simplicial sets satisfying certain conditions. Weakening these conditions into a common generalized \textit{weak Kan condition} leads to a model of $\infty$-categories building on simplicial sets: so called \textit{quasi-categories} or \textit{weak Kan complexes}. See \cite[\href{https://kerodon.net/tag/003A}{Tag 003A}]{Kero} for a formal definition based on the theory by Andr\'e Joyal. We refer to $0$-simplices of the quasi-category as \textit{objects} and to $1$-simplices as \textit{morphisms}. We define $n$-morphisms recursively as morphisms of $n-1$-morphisms.
\end{reminder}
\begin{reminder}{\textbf{Kan complexes}}\label{kancompl}\\
 All weak Kan complexes satisfying the same \textit{Kan condition} as the singular simplicial complexes of topological spaces are called \textit{Kan complexes}. They represent a model of $\infty$-groupoids, i.e. $\infty$-categories in which all $n$-morphisms (for $n\geq 1$) are equivalences. Moreover, they are determined by topological spaces up to weak equivalence. This is also known as \textit{homotopy hypothesis} induced by the adjunction 
\[
 \begin{tikzcd}[column sep = huge]
            \sSet \arrow[r, shift left=1ex, "|\cdot|"] & \Top, \arrow[l, shift left=.5ex, "\mathrm{Sing}(\cdot)"]
           \end{tikzcd}
    \]
where the \textit{singular simplicial complex} construction $\mathrm{Sing}(\cdot)$ maps into the category $\Kan$ of Kan complexes and the \textit{geometric realization} $|\cdot|$ into the category $\CW$ of CW-spaces.
\end{reminder}
By the definition through quasi-categories, notions for $\infty$-categories can be defined in terms of simplicial sets.
\begin{definition}
Let $\cat$ and $\mathcal{D}$ be $\infty$-categories.
\begin{itemize}
\item[1.] An \textit{$\infty$-functor} $F\colon \cat \rightarrow \mathcal{D}$ is defined as a morphism of simplicial sets.
\item[2.] An $\infty$-functor is an \textit{equivalence of $\infty$-categories} if it induces an equivalence of simplicial sets. 
\end{itemize}
\end{definition}
The notion of an isomorphism of objects in a category will be replaced by an equivalence of objects in an $\infty$-category as defined in the subsequent definition.
\begin{definition}{\textbf{Equivalences in $\infty$-categories}}\\
Let $\cat$ be an $\infty$-category. Two objects $X,Y \in \cat$ are called \textit{equivalent} if there exist morphisms $f\colon X \rightarrow Y$ and $g\colon Y \rightarrow X$ and (necessarily invertible) $2$-morphisms  $g\circ f \rightarrow id_X$ and $f\circ g \rightarrow id_Y$.
\end{definition}
Many familiar concepts in ordinary category theory translate to $\infty$-categories in a straight-forward way. In this setting, the $\infty$-category $\Ani$ of animated sets takes on the role of the category $\Set$. This $\infty$-category will be discussed in more detail in Chapter \ref{animation}. \\
The following table shall give an insight on how concepts can be translated:
\begin{table}[H]
\centering
\begin{tabular}{|l|cc|}
\hline
         & \textbf{Categories}      & \textbf{$\infty$-Categories}     \\ \hline
\textbf{Morphisms} &    &  \\ \hline
\textbf{$\Hom_{\cat}(X,Y)$} &  Hom-set: $\Hom_{\cat}(X,Y) \in \Set$              &   Hom-space: $\Hom_{\cat}(X,Y) \in \Ani$            \\ \hline
\textbf{$X\cong Y$} &   Isomorphism             &    Equivalence           \\ \hline
\textbf{Functors} &     & \\ \hline
\textbf{$\Fun(\cat, \mathcal{D})$} & Category of functors & $\infty$-Category of $\infty$-functors\\ \hline
\textit{Fully faithful} & \multicolumn{2}{c|}{\textit{$\Hom_{\cat}(X,Y)\cong \Hom_{\cat}(FX,FY)$}} \\ \hline
\textit{Representable} & \multicolumn{2}{c|}{\textit{$\Hom_{\cat}(X,\cdot)\cong F(\cdot)$}} \\ \hline
\textit{Adjoint} & \multicolumn{2}{c|}{\textit{$\Hom_{\cat}(FX,Y)\cong \Hom_{\cat}(X,GY)$}} \\ \hline
\textbf{Yoneda embed.} &    $\cat \rightarrow \Fun(\cat^{\mathrm{op}},\Set)$ &  $\cat \rightarrow \Fun(\cat^{\mathrm{op}},\Ani)$         \\ \hline
\textbf{Uniqueness}& up to a unique isomorphism & up to a contractible choice \\ \hline
\end{tabular}
\end{table}
As in ordinary category theory, one can express the notion of an equivalence of $\infty$-categories in terms of fully faithfulness and essential surjectivity.
\begin{definition and proposition}
An $\infty$-functor $F: \cat \rightarrow \mathcal{D}$ of $\infty$-categories is called essentially surjective if for every object $y \in \mathcal{D}$, there is an object $x \in \cat$ and an equivalence $F(x)\rightarrow y$ in $\mathcal{D}$.\\ An $\infty$-functor between $\infty$-categories is an equivalence if and only if it is fully faithful and essentially surjective.
\end{definition and proposition}
\begin{proof}
This is \cite[\href{https://kerodon.net/tag/01JX}{Tag 01JX}]{Kero}.
\end{proof}
\begin{remark}{\textbf{Limits and colimits}}\\
If $\cat$ is an $\infty$-category and $K$ is a simplicial set, we say that a map $d\colon K \rightarrow \cat$ is a diagram indexed by $K$ in $\cat$. The diagram is called small whenever $K$ is a small simplicial set. One can define generalized notions of limits and colimits which are uniquely defined up to a contractible choice and which will be denoted as usual by \textit{lim} and \textit{colim}. We say that a (large) $\infty$-category $\cat$ is complete if every (small) diagram $K\rightarrow \cat$ has a limit. There is a dual notion of a cocomplete $\infty$-category. One can speak of functors preserving (co)limits. In the same way, notions like \textit{(co)filtered (co)limits, (co)continuous functors and (co)compact objects} can be transferred to the $\infty$-categorical setting.
\end{remark}
\begin{proposition}
Let $F\colon \cat \rightarrow \mathcal{D}$ be an $\infty$-functor between $\infty$-categories
which has a right adjoint $F\colon \cat \rightarrow \mathcal{D}$. Then $F$ preserves all (small) colimits which exist in $\cat$, and $G$ preserves all (small) limits which exist in $\mathcal{D}$.
\end{proposition}
\begin{proof}
See \cite[Proposition 5.2.3.5.]{Lurie2009}.
\end{proof}
\begin{warning}{\textbf{Terminology and notation}}\label{term}\\
From Chapter \ref{animation} on, we will work with $\infty$-categories. The preceding results lead to the following justification: Whenever dealing with $\infty$-categories, we will treat them like ordinary categories, meaning we will argue in proofs with arguments like the Yoneda embedding, adjointness of functors, equivalence of $\infty$-categories etc. just like in the ordinary setting. If we work with an ordinary category in the context of $\infty$-categories, we will identify it with its nerve without making this explicit. Moreover, we will omit the addition of "$\infty$" if it is clear that all objects are considered in the $\infty$-categorical setting.
\end{warning}
There are three more notions which will appear in the Chapters \ref{animation}-\ref{homotopy}.
\begin{definition}{\textbf{Accessible and presentable}}
\begin{itemize}
\item[1.] 
Let $\kappa$ be a regular cardinal. An $\infty$-category $\cat$ is $\kappa$-accessible if it admits small $\kappa$-filtered colimits and contains an essentially small full sub-$\infty$-category which consists of $\kappa$-compact objects and generates $\cat$ under small $\kappa$-filtered colimits. We say that $\cat$ is \textit{accessible} if it is $\kappa$-accessible for some regular cardinal $\kappa$. 
\item[2.] An $\infty$-category $\cat$ is \textit{presentable} if it is accessible and
admits small colimits.
\item[3.]
If $\cat$ is an accessible $\infty$-category, then an $\infty$-functor $F\colon \cat \rightarrow \mathcal{D}$ is \textit{accessible} if it preserves $\kappa$-filtered colimits for some regular cardinal $\kappa$ (called $\kappa$-continuous).
\end{itemize}
\end{definition}
Rougly speaking, an $\infty$-category is accessible if it is generated under small $\kappa$-filtered colimits by a small sub-$\infty$-category.

\chapter{Condensed sets}\label{condset}
This chapter introduces the main concept of condensed sets and states basic results. Note Warning \ref{ignorsetiss} for a comment about set-theoretic problems that come along when defining condensed sets in the way we do and how to deal with them. For further reading and a deeper insight in the concept of condensed sets, I suggest the reader to have a look at the lecture notes \cite{Scholze2019} and \cite{Scholze2019a} which present results of the joint work by Dustin Clausen and Peter Scholze. These lecture notes are the fundamental sources of the following sections. More on the theory of \textit{Condensed mathematics} was presented by Clausen and Scholze in several sessions \cite{Clausen, Scholze2019b}.
\section{Equivalent definitions of condensed sets}
\begin{definition}{\textbf{Condensed sets}}\label{condsetdef}\\
A \textit{condensed set} is a sheaf on the site $\PFin$ of profinite sets with coverings given by finite families of jointly surjective maps as introduced in Example \ref{examplepretopo}. \\
In other words, with regards to Example \ref{cohcompH} and Example \ref{sheavcoh}, a condensed set is a functor 
\begin{align*}
T \colon \PFin^{\text{op}} &\rightarrow \Set\\
S&\mapsto T(S)
\end{align*}
satisfying $T(\emptyset)=\ast$ and the following two sheaf conditions:
\begin{enumerate}
\item[(I)] For any profinite sets $S_1, S_2$, the natural map
$$T(S_1 \amalg S_2) \rightarrow T(S_1) \times T(S_2)$$
is a bijection.
\item[(II)] For any surjection $S^{\prime}\rightarrow S$ of profinite sets with the fibre product $S^{\prime} \times_S S^{\prime}$ and its two projections $p_1, p_2$ to $S^{\prime}$, the map
$$T(S)\rightarrow\{x\in T(S^{\prime})\ |\ T(p_1)(x)=T(p_2)(x) \in T(S^{\prime} \times_S S^{\prime})\}$$
is a bijection.
\end{enumerate}
\end{definition}
Note that condition (II) is just a reformulation of the equalizer diagram of Example \ref{sheavcoh} in the category of sets.
The definition of condensed sets can be modified by passing to other underlying categories. We have the following:
\begin{proposition}\label{condsetequi}
The category of condensed sets is equivalent to:
\begin{enumerate}
\item[(i)] The category of sheaves on the site $\CoHa$ of compact Hausdorff spaces  with coverings given by finite families of jointly surjective maps. This the category of contravariant functors
\begin{align*}
T \colon \CoHa^{\mathrm{op}} &\rightarrow \Set\\
S&\mapsto T(S)
\end{align*}
satisfying $T(\emptyset)=\ast$ and (I),(II) for compact Hausdorff spaces.
\item[(ii)] The category of contravariant functors on the category $\extdis$ of extremally disconnected profinite sets
\begin{align*}
T \colon \extdis^{\mathrm{op}} &\rightarrow \Set\\
S&\mapsto T(S)
\end{align*}
satisfying $T(\emptyset)=\ast$ and (I) for extremally disconnected profinite sets. 
\end{enumerate}
The functors inducing the equivalences are given by restriction and by right Kan extension.
\end{proposition}
\begin{proof}
The category $\extdis$ of extremally disconnected profinite sets is a full subcategory of the category of profinite sets $\PFin$ (resp. of compact Hausdorff spaces $\CoHa$). By Proposition \ref{Surjext}, every compact Hausdorff space admits a surjection from an extremally disconnected compact Hausdorff space. Thus by definition of the Grothendieck (pre)topology on the site $\PFin$ (resp. $\CoHa$ in (i)), the category of extremally disconnected profinite sets is a basis of this site (cf.\ Def.\ \ref{basis}). With regards to Proposition \ref{basistopo}, one obtains an equivalence of sheaf categories by restriction from profinite sets (resp. compact Hausdorff spaces). The reverse direction of the equivalence is given by taking the right Kan extension of the restriction as stated in Proposition \ref{rightKan}. 
As any covering of extremally disconnected profinite sets splits, condition (II) drops  by passing to the restriction and can thus be omitted in $(ii)$. Indeed, for every surjection $f \colon S^{\prime} \rightarrow S$ of extremally disconnected profinite sets, using Proposition \ref{characextdis} we can find  a section $g \colon S \rightarrow S^{\prime}$ with $f \circ g =\mathrm{id}_S$. Then for a condensed set $T$ by $T(g)\circ T(f)=id_{T(S)}$, the map $T(f)\colon T(S)\rightarrow T(S^{\prime})$ is injective. Moreover, its image is contained in $$\{x\in T(S^{\prime})\ |\ T(p_1)(x)=T(p_2)(x) \in T(S^{\prime} \times_S S^{\prime})\}$$ since $f\circ p_1=f\circ p_2$. It remains to show that $T(f)$ maps surjectively onto this set. Let $$x\in \{x\in T(S^{\prime})\ |\ T(p_1)(x)=T(p_2)(x) \in T(S^{\prime} \times_S S^{\prime})\},$$ then $$T((g\circ f)\times_{S}\mathrm{id}_{S^{\prime}})(T(p_1)(x))=T((g\circ f)\times_{S}\mathrm{id}_{S^{\prime}})(T(p_2)(x))$$ which implies $T(f)(T(g)(x))=x$. Hence, we found a preimage for $x$.
This proves the proposition.
\end{proof}
\begin{remark}
Particularly, the previous proposition states that any condensed set $T$ (resp. contravariant functor in (i)) is already determined by its values on extremally disconnected profinite sets. Furthermore, it suffices to consider equalizer diagrams: Let $K$ be a compact Hausdorff space and $S, S^{\prime}$ are extremally disconnected profinite sets given as in Theorem \ref{coeqext}. As we know that $S\rightarrow K$ is a surjection of compact Hausdorff spaces, one obtains an equalizer diagram
\[\begin{tikzcd}[column sep = huge]
T(K) \ar[r]
&
T(S) \ar[r,shift left=.75ex, "T(\tilde{p}_1)"]
  \ar[r,shift right=.75ex,swap, "T(\tilde{p}_2)"]
&
T(S\times_K S)
\end{tikzcd}
\] 
using sheaf condition (II). We have a similar equalizer diagram for the surjection $f\colon S^{\prime} \rightarrow S\times_K S$ which, in particular, induces injectivity of the map of sets $T(f)\colon T(S\times_K S) \rightarrow T(S^{\prime})$. Composition of this map with the equalizer diagram above yields again an equalizer diagram
\[\begin{tikzcd}[column sep = huge]
T(K) \ar[r]
&
T(S) \ar[r,shift left=.75ex, "T(f)\circ T(\tilde{p}_1)"]
  \ar[r,shift right=.75ex,swap, "T(f)\circ T(\tilde{p}_2)"]
&
T(S^{\prime})
\end{tikzcd}
\] 
by injectivity of $T(f)$.
\end{remark}
We will work with the characterization of $\Cond$ as in Proposition \ref{condsetequi} (ii) since it allows us to treat condensed sets like presheaves which send finite disjoint unions to finite products. 
\begin{remark}\label{generdefcond}
Similarly to Definition \ref{condsetdef} and Proposition \ref{condsetequi}, respectively, one can define condensed rings/groups/\dots \ as contravariant functors into the corresponding category or, equivalently, as ring/group/\dots \ obejct in the category of condensed sets.
\end{remark}
\begin{warning}{\textbf{Set-theoretic problems}}\label{ignorsetiss}\\
The definition of condensed sets by one of the equivalent characterizations has some set-theoretic problems because the categories of compact Hausdorff spaces, profinite sets and extremally disconnected profinite sets are large and hence the considered functor categories are no longer locally small. In \cite{Scholze2019}, one can find a way of resolving this issue. In a nutshell, in the case of Proposition \ref{condsetequi} (ii) one should add the requirement "[...] such that for some uncountable strong limit cardinal $\kappa$, it is the left Kan extension of its restriction to $\kappa$-small extremally disconnected sets." as in \cite[Definition 2.11.]{Scholze2019}. In other words, we should only consider small functors as mentioned in Warning \ref{warn}. Throughout this work, the problem will be ignored in most cases as it does not play a central role for our purposes. 
\end{warning}
\section{Topological spaces as condensed sets}
In this section, we investigate how the category $\Cond$ of condensed sets can be seen as replacement for the category $\Top$ of topological spaces.
\begin{example}\label{extop}
To every topological $T_1$-space $T$, we can associate a condensed set $\underline{T}$ by the assignment
\begin{align*}
\underline{T}\colon\extdis^{\mathrm{op}} &\rightarrow \Set\\
S &\mapsto \Hom_{\Top}(S,T).
\end{align*}
This clearly satisfies the conditions in Proposition \ref{condsetequi} (ii). If we do not assume $T$ to be $T_1$, there are set-theoretical issues \cite[Warning 2.14.]{Scholze2019}. Therefore, we will implicitly make this assumption.
\end{example}
\begin{corollary}\label{exttoall}
The extension of $\underline{T}$ to all compact Hausdorff spaces through right Kan extension as in Proposition \ref{condsetequi} also is given by the formula
\begin{align}\label{formula} 
S &\mapsto \Hom_{\Top}(S,T)
\end{align}
with $S$ an object of the category of all compact Hausdorff spaces. 
\end{corollary}
\begin{proof}
For any surjection $q\colon K^{\prime}\rightarrow K$ of compact Hausdorff spaces, the diagram 
\[
\begin{tikzcd}
K^{\prime}\times_K K^{\prime} \ar[r,shift left=.75ex,"p_1"]
  \ar[r,shift right=.75ex,swap,"p_2"]
&
K^{\prime} \ar[r,"q"] 
&
K 
\end{tikzcd}
\]
is a coequalizer (cf.\ Ex.\ \ref{cohcompH}). Combining this with the natural isomorphism of sets
$$\mathrm{lim}_{i \in I}\Hom_{\Top}(B_i, A)=\Hom_{\Top}(\colim_{i\in I} B_i, A)$$ induces the bijection of sheaf condition (II). Hence, the presheaf defined by (\ref{formula}) is a condensed set for the characterization provided in Proposition \ref{condsetequi} (i). As its restriction to $\extdis^{\mathrm{op}}$ is clearly given by $\underline{T}$, it has to coincide with the unique extension of $\underline{T}$ to all compact Hausdorff spaces constructed through right Kan extension.
\end{proof}
The functor $T\mapsto \underline{T}$ involves more information about the relation between topological spaces and condensed sets, which is collected in the following proposition. Above all, the equivalence in the third statement has a key function for the next chapter.
\begin{proposition}\label{equivcomphaus}
\hspace{2em}
\begin{enumerate}
\item[1.] The functor $T \mapsto \underline{T}$ has a left adjoint $X\mapsto X(\ast)_{\mathrm{top}}$ sending any condensed set $X$ to its underlying set $X(\ast)$ equipped with the quotient topology given by the map
$$\amalg_{S, a\in X(S)} S\rightarrow X(\ast).$$
\item[2.] The functor $T\mapsto \underline{T}$ restricted to compactly generated topological spaces is fully faithful.
\item[3.] The functor $T\mapsto \underline{T}$ induces an equivalence between the category of compact Hausdorff spaces and the category of quasicompact quasiseparated condensed sets.
\end{enumerate}
\end{proposition}
\begin{proof}
This is \cite[Proposition 1.2.]{Scholze2019a} and stated with proof in \cite[Proposition 1.9.]{Scholze2019}.
\end{proof}
\begin{remark}\label{topascondset}
As we know that $T\mapsto \underline{T}$ sends a ($T_1$-)space on a condensed set, we can justify the following: Speaking of a compactly generated topological space in the context of condensed sets, we will usually mean the condensed set associated to the topological space under the fully faithful embedding of Proposition \ref{equivcomphaus}.
\end{remark}
\begin{remark}
The notions of quasicompactness and quasiseparatedness of Definition \ref{quasi} unwind for condensed sets as the following:
A condensed set $X$ is quasicompact if there is some profinite $S$ with a surjective map $S \rightarrow X$. A condensed set X is quasiseparated if for any two profinite sets $S_1, S_2$  with maps to $X$, the fibre product $S_1 \times_X S_2$ is quasicompact.
\end{remark}

\chapter{Compact projective objects}\label{compro}
In the fourth chapter, we want to define the so called \textit{animation} of a category - an $\infty$-category freely generated under sifted colimits by some special subcategory. The subcategory is given by the compact projective objects. At the end of this chapter, the compact projective objects in $\Cond$ will be determined in an explicit way. We start with the corresponding definitions taken from \cite{Adamek2011} and \cite{Adamek2010}.
\section{Definitions of compact and projective}
In the following, let $\cat$ be a category that admits all (small) colimits. Indeed, the category $\Cond$ satisfies this assumption \cite[Comment after Remark 2.13.]{Scholze2019}.
\begin{definition}{\textbf{Projectivity notions}}\\
An object $X\in \cat$ is called
\begin{enumerate}
\item[(i)] \textit{regular projective} if $\Hom_{\cat}(X, \cdot)$ commutes with regular epimorphisms, i.e. coequalizers of parallel arrows $Y \rightrightarrows Z$.
\item[(ii)] \textit{projective} if $ \Hom_{\cat}(X,\cdot)$ commutes with reflexive coequalizers, i.e. coequalizers of parallel arrows $Y \rightrightarrows Z$ with
a simultaneous section $Z \rightarrow Y$ of both maps.
\end{enumerate}
\end{definition}
\begin{definition}{\textbf{Compact objects}}\\
An object $X\in \cat$ is called \textit{compact} (or \textit{finitely presentable}) if $\Hom_{\cat}(X, \cdot)$ commutes with filtered colimits, i.e. colimits over filtered categories.
\end{definition}
%
The subsequent propositions state how the notions above are related and collect some stability properties we will need later on. 
\begin{definition}
A \textit{kernel pair} in a category is a pair of morphisms $R\rightrightarrows X$ forming the pullback of a given morphism $f\colon X\rightarrow Y$ along itself, i.e. the diagram
\[
\begin{tikzcd}
 R\ar[r]\ar[d] & X \ar[d, "f"] \\
 X \ar[r, "f"] & Y
\end{tikzcd}
\]
is a pullback square.
\end{definition}
\begin{proposition}\label{stabimpl}
All regular projective objects are projective. In a category with kernel pairs, also the other implication is true. 
\end{proposition}
\begin{proof}
The first statement is clear by definition and the second follows as in a category with kernel pairs, every regular epimorphism is a reflexive coequalizer by \cite[Example 3.4]{Adamek2011}.
\end{proof}
\begin{proposition}{\textbf{Stability properties}}\label{stabprop}\\
One has the following closedness properties for the notions defined above.
\begin{enumerate}
\item[(i)]Compact objects are closed under finite colimits.
\item[(ii)]Regular projective objects are closed under coproducts.
\item[(iii)] The notions compact and regular projective
are closed under retracts.
\end{enumerate}
\end{proposition}
\begin{proof}
See \cite[Lemmata 5.9, 5.11 and 5.13]{Adamek2011}.
\end{proof}
\begin{example}\label{exampleext}
Extremally disconnected spaces are regular projective in the category of compact Hausdorff spaces: In Example \ref{cohcompH}, we already argued that every epimorphism of compact Hausdorff spaces is effective. As the notions of regular and effective coincide in a category with fibre products, every epimorphism in this category is regular. And by Proposition \ref{characextdis} (iii), the functor $\Hom_{\Top}(S,\cdot)$ of some extremally disconnected compact Hausdorff space $S$ preserves all epimorphisms. As Proposition \ref{characextdis} actually states an equivalence, even the converse is true: The extremally disconnected spaces are precisely the regular projective compact Hausdorff spaces. This is also a theorem by Gleason stated in \cite[III 3.7.]{Johnstone1982}.
\end{example}
\section{Compact projective objects in \texorpdfstring{$\Cond$}{the category of condensed sets}}
In this section, we will prove that compact projective objects in $\Cond$ are given by extremally disconnected compact Hausdorff spaces. For this purpose, we will investigate in more detail what compact projective condensed sets look like. Hereby, the results of the main Proposition \ref{retandfincol} can be formulated in the context of an arbitrary presheaf category. Or even more generally, they can be translated to all categories where objects are given as colimits of some specific subcategory. Recall that we consider condensed sets as presheaves over $\extdis$ satisfying additional properties as in \ref{condsetequi} (ii).\\
For the next proposition, we introduce the notion of $1$-sifted colimits. Its $\infty$-categorical analogue will appear in Chapter \ref{animation}.\begin{definition}
A small category $\cat$ is called \textit{1-sifted} if $\cat$-colimits commute with finite products in $\Set$. That is, for a given diagram
$$F\colon \cat \times I \rightarrow \Set,$$
where $I$ is a finite discrete category, the canonical morphism
$$ \colim_{c \in \cat} \prod_{i \in I} F(c,i) \rightarrow \prod_{i \in I} \colim_{c \in \cat} F(c,i)$$
is an isomorphism. Colimits over $1$-sifted categories are called \textit{$1$-sifted colimits}.
\end{definition}
Every reflexive coequalizer is a sifted colimit and classic examples for $1$-sifted categories are filtered categories. 
\begin{proposition}\label{extdiscompro}
All extremally disconnected compact Hausdorff spaces are compact and projective as objects in $\Cond$.
\end{proposition}
\begin{proof}
Let $S$ be an extremally disconnected compact Hausdorff space. By the Yoneda lemma one has 
$$\Hom_{\Cond}(\Hom_{\Top}(\cdot,S),M)\cong M(S)$$
for all $M\in \Cond$. Further, $1$-sifted colimits in the category of condensed sets can be calculated pointwise: In a presheaf category, all colimits can be computed pointwise. If the colimit additionally commutes with finite products in $\Set$, i.e. is $1$-sifted, the pointwise colimit of condensed sets again defines a presheaf sending finite disjoint unions to finite products, i.e. it is a condensed set. Thus, by the above, the evaluation functor 
$$\Cond\rightarrow \Set, \ M \mapsto M(S)$$
commutes with all $1$-sifted colimits in 
the category of condensed sets. By \cite[Theorem 2.1]{Adamek2010}, a functor $F\colon \Cond \rightarrow \Set$ preserves $1$-sifted colimits if and only if it preserves filtered colimits and reflexive coequalizers.
Therefore, 
$$M \mapsto \Hom_{\Cond}(\Hom_{\Top}(\cdot,S),M)$$
commutes with 
filtered colimits and reflexive coequalizers and hence $\Hom_{\Top}(\cdot,S)$ is compact projective.
\end{proof}
\begin{remark}\label{reprecom}
\hspace{2em}
\begin{enumerate}
\item Note that the argument we applied in the above proposition cannot be applied in the same manner if we work with the definition of condensed sets as sheaves of sets over (totally disconnected) compact Hausdorff spaces. More precisely, due to the sheaf condition (II) in Definition \ref{condsetdef} a colimit of sheaves calculated in the presheaf category does not have to be a sheaf again.  This is where sheafification usually plays a role. Hence, colimits cannot be calculated pointwise in an arbitrary sheaf category.
\item With the same idea as in the proof of Proposition \ref{extdiscompro} one can show that in a presheaf category every (co)representable presheaf is compact projective.
\end{enumerate}
\end{remark}
\begin{lemma}\label{colimit}
All condensed sets are (small) colimits of extremally disconnected compact Hausdorff spaces.
\end{lemma}
\begin{proof}
Every condensed set is given as presheaf on the category of extremally disconnected profinite sets. Every presheaf on a small category is a colimit of representables by \cite[Theorem 1 of Section 7 in Chapter III]{Lane1998}. In Lemma \ref{smallcolimits} we stated the same result for small functors on large categories and thus the statement follows under notice of Warning \ref{ignorsetiss}.
\end{proof}
\begin{proposition}\label{retandfincol}
\hspace{2em}
\begin{enumerate}
\item[(i)] Every regular projective condensed set is a retract of a coproduct of extremally disconnected compact Hausdorff spaces.
\item[(ii)]Every condensed set that is compact as an object in the category of condensed sets is a finite colimit of extremally disconnected compact Hausdorff spaces.
\end{enumerate}
\end{proposition}
\begin{proof}
\begin{enumerate}
\item[(i)] 
Every colimit can be expressed as a coequalizer of a diagram of coproducts as in \cite[\href{https://stacks.math.columbia.edu/tag/002P}{Tag 002P}]{Stacks}.
Hence, by Lemma \ref{colimit} for every regular projective condensed set $P=\colim_{k\in K}X_k$ with $X_k$ extremally disconnected compact Hausdorff for all $k\in K$ there exists a regular epimorphism
$$\amalg_{k\in K}X_k\rightarrow P.$$
Consider the map of sets
$$\Hom(P, \amalg_{k\in K}X_k)\rightarrow \Hom(P,P)$$ 
given by postcomposition with the regular epimorphism. Then, by projectivity of $P$ this map is also an epimorphism. Thus, the identity $P\rightarrow P$ factors through a map $P\rightarrow \amalg_{k\in K}X_k$, i.e. one has
$$\mathrm{id}\colon P\rightarrow \amalg_{k \in K} X_k\rightarrow P.$$
Particularly, by this $P$ is a retract of $\amalg_{k \in K}X_k$ and the statement follows.
\item[(ii)]
We start with a similar argument as in the previous proof, but here we have to work a bit more. Every colimit can be expressed as a filtered colimit of its finite colimits. Hence, by Lemma \ref{colimit} some condensed set $X=\colim_{k\in K}X_k$ is given as a colimit
$$X=\colim_{j \in J} (\colim_{i \in I_j} X_{ij})$$
where $J$ is filtered, all $I_j$ are finite and every $X_{ij}$ is extremally disconnected compact Hausdorff space as it coincides with some $X_k$. We assume $X$ to be compact in the category of condensed sets. Then, by the property of compactness the identity morphism $X \rightarrow X$ factors through some finite colimit $X \rightarrow \colim_{i\in I_{j^{\prime}}}X_{ij^{\prime}}$. Indeed, one has
\begin{align*}
\Hom(X,X)&=\Hom(X, \colim_{j \in J} (\colim_{i \in I_j} X_{ij}))\\&=\colim_{j \in J} \Hom(X,(\colim_{i \in I_j} X_{ij}))
\end{align*}
and gets by this a surjective map 
\begin{align*}
\amalg_{j \in J}  \Hom(X,(\colim_{i \in I_j} X_{ij})) \rightarrow \Hom(X,X),
\end{align*}
given by postcomposition with the respective natural map $\colim_{i \in I_j} X_{ij} \rightarrow X$ for all $j\in J$. Thus, we have 
\begin{align*}
\mathrm{id}\colon X \rightarrow \colim_{i \in I_{j^{\prime}}} X_{ij^{\prime}} \overset{\phi}{\longrightarrow} X
\end{align*}
for some $j^{\prime} \in J$. We want to show that $X$ is already given by this finite colimit. For this, it is enough to show that $\colim_{i \in I_{j^{\prime}}} X_{ij^{\prime}}$ satisfies the colimit property of $X=\colim_{k\in K}X_k$. We assume another condensed set $Y$ with a family of morphisms $\psi_k\colon X_k\rightarrow Y$ s.t. for every morphism $f_{kl}\colon X_k\rightarrow X_l$ with $k,l \in I$ the commutativity property $\psi_l \circ f_{kl} =\psi_l$ is fulfilled. Then, there exist unique morphisms $u$ and $u_{j^{\prime}}$ s.t for every such morphism $f_{kl}\colon X_k\rightarrow X_l$ with $k,l \in I$ and $f^{j^{\prime}}_{kl}\colon X_{kj^{\prime}}\rightarrow X_{lj^{\prime}}$ with $k,l \in I_{j^{\prime}}$, respectively, the diagrams
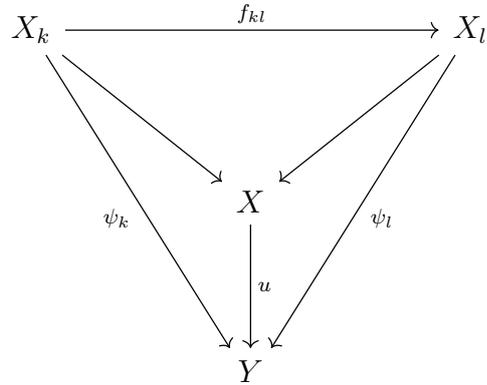
\[
\begin{tikzcd}[row sep=4em, column sep={7em,between origins}]
X_{k} \arrow[rr,"f_{kl}"] \arrow[dr,]
  \arrow[ddr,swap,end anchor={[xshift=0.2em]north west},"\psi_{k}"]
&& 
X_{l} \arrow[dl,swap,] 
  \arrow[ddl,end anchor={[xshift=-0.2em]north east},"\psi_{l}"]
\\
& X \arrow[d,"u"]
\\
& Y
\end{tikzcd}
\]
and 
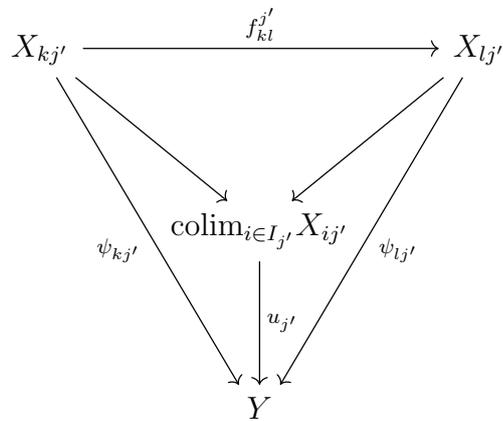
\[
\begin{tikzcd}[row sep=4em, column sep={7em,between origins}]
X_{kj^{\prime}} \arrow[rr,"f^{j^{\prime}}_{kl}"] \arrow[dr,]
  \arrow[ddr,swap,end anchor={[xshift=0.2em]north west},"\psi_{kj^{\prime}}"]
&& 
X_{lj^{\prime}} \arrow[dl,swap,] 
  \arrow[ddl,end anchor={[xshift=-0.2em]north east},"\psi_{lj^{\prime}}"]
\\
& \colim_{i \in I_{j^{\prime}}} X_{ij^{\prime}}  \arrow[d,"u_{j^{\prime}}"]
\\
& Y
\end{tikzcd}
\]
commute. Moreover, we have maps 
\begin{align*}
X_{k} \rightarrow X \overset{\phi}{\longrightarrow}  \colim_{i \in I_{j^{\prime}}} X_{ij^{\prime}}
\end{align*}
for all $k \in K$. Note that for all $i\in I_{j^{\prime}}$ the maps 
$$X_{ij^{\prime}} \rightarrow X \overset{\phi}{\longrightarrow}  \colim_{i \in I_{j^{\prime}}} X_{ij^{\prime}}$$ coincide with the natural maps $X_{ij^{\prime}}\rightarrow \colim_{i\in I_{j^{\prime}}}X_{ij^{\prime}}$ as $X_{ij^{\prime}}\rightarrow X$ factors through $X_{ij^{\prime}}\rightarrow \colim_{i \in I_{j^{\prime}}}X_{ij^{\prime}}$.
We get for all $f_{kl}\colon X_k\rightarrow X_l$ a bigger diagram 
$$\begin{tikzcd}[row sep=4em, column sep=7em]
X_k \arrow[r, "f_{kl}"] \arrow[d, ]
\arrow[to=2-2,]
& X_l \arrow[d, ]
\arrow[to=2-1, crossing over]
 \\
 \colim_{i \in I_{j^{\prime}}} X_{ij^{\prime}} \arrow[r,yshift=0.7ex,]\arrow[d,"u_{j^{\prime}}"]
& \colim_{k\in K}X_k \arrow[l, yshift=-0.7ex, "\phi" ] \arrow[to=3-1,"u"]
\\
Y
\end{tikzcd}$$
as combination of the diagrams above, where all sub-diagrams contained in the upper square naturally commute by construction. Then, the lower triangle commutes by uniqueness of $u$. Altogether, we get for all  $f_{kl}\colon X_k\rightarrow X_l$ a commutative diagram 
\[
\begin{tikzcd}[row sep=4em, column sep={7em,between origins}]
X_{k} \arrow[rr,"f_{kl}"] \arrow[dr,]
  \arrow[ddr,swap,end anchor={[xshift=0.2em]north west},"\psi_{k}"]
&& 
X_{l} \arrow[dl,swap,] 
  \arrow[ddl,end anchor={[xshift=-0.2em]north east},"\psi_{l}"]
\\
& \colim_{i \in I_{j^{\prime}}} X_{ij^{\prime}} \arrow[d,"u"]
\\
& Y
\end{tikzcd}
\]
where $u_{j^{\prime}}$ is unique as it was unique before. Thus, it is equal to $X$. This proves that every condensed set which is compact as an object in the category of condensed sets is a finite colimit of extremally disconnected compact Hausdorff spaces.
\end{enumerate}
\end{proof}
Now, by combining the previous results, we are able to determine the compact projective objects in $\Cond$.
\begin{theorem}\label{thmcompro}
The compact projective objects in the category $\Cond$ of condensed sets are exactly given by the extremally disconnected compact Hausdorff spaces.
\end{theorem}
\begin{proof}
First, we note that the condensed sets which are compact projective in $\Cond$ are exactly the compact regular projective objects in $\Cond$ by Proposition \ref{stabimpl}.
By merging the results of Proposition \ref{retandfincol}, this implies that the condensed sets which are compact projective in $\Cond$ are given by all retracts of finite coproducts of extremally disconnected compact Hausdorff spaces and thus are quasicompact and quasiseparated. More precisely, following the proof of Proposition \ref{retandfincol} (i) under assumption of (ii) the constructed coproduct turns out to be finite. Furthermore, as compact Hausdorff spaces are closed under finite coproducts and retracts by Lemma \ref{compstabprop} this is, due to the equivalence of categories in Proposition \ref{equivcomphaus} $(3.)$, also true for quasicompact quasiseparated condensed sets. The statement of the theorem follows since extremally disconnected spaces are the regular projective objects in the category of compact Hausdorff spaces as stated in Example \ref{exampleext} and this property is closed under retracts and finite coproducts by Proposition \ref{stabprop}.
\end{proof}
\begin{remark}
In the theorem, we have indeed seen that all compact projective condensed sets are quasicompact and quasiseparated.
Note that projectivity plays an important role here: One can show that all quasicompact quasiseparated condensed sets are compact in the category of condensed sets whereas the other direction is not true.\\
Indeed, by Proposition \ref{stabprop} compact objects of a category are closed under finite colimits in the category. This does not hold for compact Hausdorff spaces in the category of (compactly generated) topological spaces. One can construct a coequliazer of compact Hausdorff spaces that is not Hausdorff in the category of all (compactly generated) topological spaces. Thus, it cannot identified with a quasicompact quasiseparated condensed set under the equivalence in Proposition \ref{equivcomphaus} $(3.)$. The categories of compact condensed sets and quasicompact quasiseparated condensed sets cannot coincide, i.e. there exist compact condensed sets which are not quasicompact quasiseparated. 
\end{remark}

\chapter{The animation of condensed sets}\label{animation}
This chapter introduces the $\infty$-category $\AniCond$ - the animation of the category $\Cond$ - which combines the topological space direction of condensed sets $\Cond$ with the homotopy theory direction of animated sets $\Ani$. The following (non-commutative) diagram provides an overview of the claimed combination and where to find corresponding results. 
\begin{equation*}
\begin{tikzcd}[row sep=huge]
\Top  \arrow[rightarrow,d,xshift=0.7ex, "\ref{extop}"] &
\mathcal{CG}\Top \arrow[hookrightarrow, dl, "\ref{equivcomphaus}"] \arrow[l,"\supseteq" description]&
\CW \arrow[l,"\supseteq" description] \arrow[r,yshift=0.7ex, "\text{Sing}(\cdot)"]  \arrow[dr, "\ref{aniset}" '] &
\Kan \arrow[d, "\ref{aniset}"] \arrow[l,yshift=-0.7ex, "\vert \cdot \vert"]
\\
\boldsymbol{\Cond} \boldsymbol{\arrow[hookrightarrow,d, xshift=0.7ex, "\ref{thmAniCond}"]} \arrow[u,xshift=-0.7ex, "\ref{equivcomphaus}"]& \extdis \arrow[dl, "\ref{defani} \ \& \ \ref{thmcompro}" ] \arrow[l, "\supseteq" description]& \Set  \arrow[hookrightarrow,r,yshift=0.7ex, "\ref{truncationprop}"] & \boldsymbol{\Ani} \boldsymbol{\arrow[hookrightarrow,d, xshift=0.7ex, "\ref{thmAniCond}"]} \arrow[l,yshift=-0.7ex, "\ref{truncationprop}"] 
\\
\boldsymbol{\AniCond} \boldsymbol{\arrow[rrr,"\boldsymbol{\sim}", "\ref{thmAniCond}" ']}\boldsymbol{\arrow[u,xshift=-0.7ex, "\ref{truncationprop2}"]}& & & \boldsymbol{\CondAni} \boldsymbol{\arrow[dashrightarrow, u, xshift=-0.7ex, "\ref{proleft}"]} 
\end{tikzcd}
\end{equation*}
We will start with the definition of the $\infty$-category of animated sets and then translate to the general case of an \textit{Animation}.\\
In Section \ref{condaninew} and Section \ref{topconani} we will see that not only condensed (cf.\ Rem.\ \ref{condemb}) but also animated sets can be embedded fully faithful into the animation of condensed sets which can be identified with the $\infty$-category of condensed anima (cf.\ Def.\ \ref{condanidef}). We will describe these fully faithful embeddings in more detail before passing to Chapter $\ref{homotopy}$ which will deal with a possible option of making homotopy theory with condensed sets by constructing a "left adjoint" to the fully faithful embedding $\Ani \xhookrightarrow{} \CondAni$.
\section{The \texorpdfstring{$\infty$}{infinity}-category of animated sets}
The $\infty$-category $\Ani$ of animated sets is an example for the more general construction called \textit{animation} in \cite{Cesnavicius2019} and \cite{Scholze2019a}. It plays a central role in the higher-categorical setting as it is the proper analogue of $\Set$. Indeed, all $\infty$-categories are enriched over $\Ani$. See Section \ref{forminf} for a comment on how to work with $\infty$-categories. From now on, inter alia, we will implicitly identify an ordinary category with the associated $\infty$-category obtained by the nerve construction. More on the underlying theory of $\infty$-categories presented in this chapter can be found in \cite[\href{https://kerodon.net/tag/00VJ}{Subsections 3.2.2} and \href{https://kerodon.net/tag/01YX}{5.4.1}]{Kero} and \cite[Sections 2.3.4. and 5.5.6.-5.5.8.]{Lurie2009}.\\
\newline
The $\infty$-category of animated sets can be obtained in several ways. Some of them are collected in the following definition. 
\begin{definition}{\textbf{Animated sets}}\label{aniset}\\
An \textit{animated set} or \textit{anima} is an object, i.e. a $0$-simplex, of the $\infty$-category described by one of the following equivalent characterizations. 
\begin{enumerate}
\item[(i)] Taking the category $\Top$ of topological spaces and inverting weak equivalences, universal for all $\infty$-categories (\textit{simplicial localization}).\\ \cite[Section 8]{Hoermann2019}
\item[(ii)] Taking the category $\sSet$ of simplicial sets and inverting weak equivalences, universal for all $\infty$-categories (\textit{simplicial localization}). \\ \cite[Section 8]{Hoermann2019}
\item[(iii)] Taking the simplicial (also called homotopy coherent) nerve of the category $\Kan$ of Kan complexes regarded as a simplicial category.\\ \cite[Definition 1.2.16.1.]{Lurie2009}
\item[(iv)] Taking the topological nerve of the category $\CW$ of CW complexes with continuous functions regarded as a topological category.\\ \cite[Remark 1.2.16.3.]{Lurie2009}
\item[(v)] The \textit{Animation} (cf.\ Def.\ \ref{defani}) of the category $\Set$ of sets where the compact projective sets are given by $\Set^{\mathrm{cp}}=\Set^{\mathrm{fin}}$.\\ \cite[Example 11.5.]{Scholze2019a}
\end{enumerate}
The $\infty$-category is denoted by $\Ani(\Set)$ or just $\Ani$ and is called \textit{$\infty$-category of spaces} or \textit{of $\infty$-groupoids} in standard language.
\end{definition}
\begin{remark}
As described in \cite[Theorem 1.2.]{Hoermann2019}, respectively in \cite[\href{https://kerodon.net/tag/01Z4}{Tag 01Z4}]{Kero}, the equivalence of the $\infty$-categories obtained by the constructions of (i) and (ii), respecively (iii) and (iv), is induced by the adjunction
\[
 \begin{tikzcd}[column sep = huge]
            \sSet \arrow[r, shift left=1ex, "|\cdot|"] & \Top, \arrow[l, shift left=.5ex, "\mathrm{Sing}(\cdot)"]
           \end{tikzcd}
   \]
which induces an equivalence on homotopy categories. The statement that these yield the same $\infty$-category as (v) can be found in \cite[Example 5.5.8.24.]{Lurie2009}.
\end{remark}
$\Ani$ is the classical $\infty$-topos to do homotopy theory with. 
Regarding it modeled by Kan complexes, i.e. viewing its objects as Kan complexes, yields the following notion of \textit{simplicial homotopy groups} which are defined in a similar way as for topological spaces.
This definition provides a reduction of the definition of \textit{categorical homotopy groups} for $\infty$-topoi. 
\begin{definition}{\textbf{Simplicial homotopy groups}}\label{simphomgr}\\
Let $A \in \Ani$.
\begin{enumerate}
\item The set $\pi_0(A)$ of \textit{path components} is the quotient set by the equivalence relation of being in the same path component, whereby a path between $a, b \in A$ is defined as a morphism $p \in \mathrm{Hom}_{\sSet}(\triangle^1,A)$ with initial point $p[0]=a$ and terminal point $p[1]=b$.
\item For $n\geq 1$ and $a\in A$ (meaning a map $\ast \rightarrow A$) the n\textit{th simplicial homotopy group} is the set
$$ \pi_n(A,a)\coloneqq [\triangle^{n}/\delta \triangle^{n}, A]_{*}$$
of pointed equivalence classes of morphisms 
$\triangle^{n}/\delta \triangle^{n} \rightarrow A$
with respect to $a$, i.e. morphisms that take the boundary $\delta \triangle^{n}$ of the simplicial $n$-simplex $\triangle^n$ to $a$, where the equivalence relation is defined by existence of a simplicial homotopy.
\end{enumerate}
\end{definition}
\begin{remark}\label{simphomgrre}
For a topological space $X$ the topological homotopy groups coincide with those of the singular simplicial complex $\mathrm{Sing}(X)$ \cite[\href{https://kerodon.net/tag/00VR}{Tag 00VR}]{Kero}. As in the setting of topological spaces, one can define a group structure if $n\geq 1$ and receives long exact sequences of homotopy groups by Kan fibrations and Whitehead's theorem for Kan complexes. Moreover, the groups are abelian for $n\geq 2$. See \cite[\href{https://kerodon.net/tag/00V2}{Tag 00V2}]{Kero} as a reference for these results.
\end{remark}
\begin{definition}{\textbf{Truncation in $\Ani$}}\label{truncationdef}\\
An anima is defined to be \textit{$n$-truncated} if $\pi_i(A,a)=0$ for all $a\in A$ and $i>n$. In standard language, such an anima is called an \textit{$n$-groupoid}.
\end{definition}
\begin{definition and proposition}{\textbf{Truncation functors}}\label{truncationprop}\\
There is a fully faithful embedding $v\colon \Set \hookrightarrow \Ani$ that identifies $\Set$ with the full subcategory of $0$-truncated objects of $\Ani$. The connected component functor $\pi_0 \colon \Ani \rightarrow \Set$ is left adjoint to the embedding. More generally, the embedding of $n$-truncated anima into all anima has an accessible left adjoint $\tau_{\leq n}\colon \Ani \rightarrow \Ani$.
\end{definition and proposition}
\begin{proof}
There is a fully faithful embedding $\Set \hookrightarrow \sSet$ given by the constant simplicial set construction which identifies $\Set$ with the category of discrete simplicial sets by \cite[\href{https://kerodon.net/tag/00G0}{Tag 00G0}]{Kero}. Moreover by \cite[\href{https://kerodon.net/tag/00GS}{Tag 00GS}]{Kero}, the connected component functor $\pi_0\colon \sSet \rightarrow \Set$ is left adjoint to this embedding. Actually, the construction of constant simplicial sets determines a fully faithful embedding to the category of Kan complexes by \cite[\href{https://kerodon.net/tag/00HB}{Tag 00HB}]{Kero}. Passing to (homotopy coherent) nerves, we obtain a fully faithful embedding $\Set \hookrightarrow \Ani$ of $\infty$-categories as described in \cite[\href{https://kerodon.net/tag/01Z2}{Tag 01Z2}]{Kero} under use of \cite[\href{https://kerodon.net/tag/00KZ}{Tag 00KZ}]{Kero}. By \cite[Example 2.3.4.17.]{Lurie2009}, a Kan complex is $0$-truncated if and only if it is homotopy equivalent to a discrete simplicial set. This proves the first statement. The general result is given by \cite[Proposition 5.5.6.18.]{Lurie2009}.
\end{proof}
%
%
\section{Animation of a category}
The notion of sifted colimits we will use to define the \textit{animation of a category} is probably unfamiliar to many readers. It provides a natural extension of the definition in ordinary category theory, although this is not obvious. To be complete, we will state the definition here shortly. In fact, it is not necessary to understand it in detail in order to follow the subsequent results.
\begin{definition}{\textbf{Sifted colimits}}\\
A simplicial set $K$ is \textit{sifted} if it satisfies the
following conditions:
\begin{itemize}
\item[1.] The simplicial set $K$ is nonempty.
\item[2.] The diagonal map $K\rightarrow K \times K$ is cofinal.
\end{itemize}
Correspondingly, a \textit{sifted colimit} in an $\infty$-category is a colimit of a diagram $K\rightarrow \cat$ with $K$ sifted.
\end{definition}
For the following definition, recall the notions of compact and projective defined in Chapter \ref{compro}.
\begin{definition}{\textbf{Animation}}\label{defani}\\
Let $\cat$ be an ordinary cocomplete category, i.e. it admits all small colimits,  generated under small colimits by the subcategory $\cat^{\mathrm{cp}}$ of compact projective objects. The \textit{animation} of $\cat$ is the $\infty$-category $\AniC$ freely generated under sifted colimits by $\cat^{\mathrm{cp}}$. It can be defined as the full-$\infty$-subcategory of
$$\mathrm{Fun}((\cat^{\mathrm{cp}})^{\mathrm{op}}, \Ani),$$ generated under sifted colimits by the Yoneda image.
\end{definition}
\begin{proposition}\label{presen}
For every category $\cat$ as above, the animation $\AniC$ is presentable, complete and cocomplete. 
\end{proposition}
\begin{proof}
With regards to \cite[Theorem 5.5.1.1., Proposition 5.5.8.10 (1)]{Lurie2009}, the $\infty$-category $\AniC$ is presentable and thus by \cite[Corollary 5.5.2.4]{Lurie2009} complete and cocomplete.
\end{proof}
\begin{remark}\label{inftopo}
The $\infty$-category $\Ani$ has more nice properties. As mentioned before, it is an $\infty$-topos and one can consider any $\infty$-topos as an $\infty$-category which "looks like" the
$\infty$-category of anima, just as an ordinary topos is a category which "looks
like" the category of sets. Have a look at \cite[Chapter 6]{Lurie2009} for more insight on $\infty$-topoi. As in every $\infty$-topos, filtered colimits commute with finite products in $\Ani$ by \cite[Lemma 5.5.8.11. and Remark 5.5.8.12.]{Lurie2009}.
\end{remark}
\begin{definition}{\textbf{Truncation in $\infty$-categories}}\\
An \textit{$n$-truncated} object in a general $\infty$-category is an object s.t. all mapping spaces into it, which are naturally objects in $\Ani$, are $n$-truncated as in Definition \ref{truncationdef}. If all mapping spaces are $(n-1)$-truncated, the $\infty$-category is called an \textit{$n$-category}.\end{definition} 
\begin{remark}
\begin{itemize}
\item[1.]In particular, $1$-categories as in the definition above coincide with the usual notion of a category. Every ordinary category defines a $1$-category as all mapping spaces are just sets by Proposition \ref{truncationprop}. The full claim can be obtained by combining the results of \cite[Proposition 2.3.4.5. and 2.3.4.18.]{Lurie2009}.
\item[2.] By \cite[Remark 5.5.6.4]{Lurie2009}, for the $\infty$-category $\Ani$ this more general definition of truncation coincides with the Definition \ref{simphomgr}.
\end{itemize}
\end{remark}
There is a general version of Proposition \ref{truncationprop} based on the truncation functors $\tau_{\leq i} \colon \Ani \rightarrow \Ani$.
\begin{proposition}\label{truncationprop2}
There is a fully faithful embedding $\cat \hookrightarrow \AniC$ that identifies $\cat$ with the full subcategory of $0$-truncated objects of $\AniC$. Its left adjoint is given by composition of $(\cat^{\mathrm{cp}})^{\mathrm{op}} \rightarrow \Ani$ in $\Ani(\cat)$ with the connected component functor $\pi_0 \colon \Ani \rightarrow \Set$. More generally, composition with the truncation functor $\tau_{\leq n} \colon \Ani \rightarrow \Ani$ induces a truncation $ \tau_{\leq n} \colon \AniC\rightarrow \AniC$ that is an accessible left adjoint to the embedding of the full subcategory of $n$-truncated objects of $\AniC$.
\end{proposition}
\begin{proof}[Proof]
See \cite[Third paragraph of 5.1.4.]{Cesnavicius2019}.
\end{proof}
Besides the $\infty$-category of animated sets, the animation of the category of condensed sets will play an important role throughout this work. Hence, we will put the definition down in writing here.
\begin{definition}{\textbf{Animation of condensed sets}}\label{anicond}\\
By Definition \ref{defani} and Chapter $3$, the $\infty$-category $\AniCond$ of animated condensed sets is freely generated under sifted colimits by the category $\extdis$ of extremally disconnected profinite sets. Thus, it can be defined as the full-$\infty$-subcategory of
$$\mathrm{Fun}(\extdis^{\mathrm{op}}, \Ani)$$ generated under sifted colimits by the Yoneda image.
\end{definition}
\begin{remark}\label{condemb}
By Proposition \ref{truncationprop2}, there is a fully faithful embedding $$\Cond \hookrightarrow \AniCond,$$ which admits a left adjoint given by composition with $\pi_0\colon \Ani \rightarrow \Set$.
\end{remark}

\chapter{Pro-\texorpdfstring{$\infty$}{infinity}-categories}\label{procatbase}
\section{Pro-categories in the \texorpdfstring{$\infty$}{infinity}-setting}
In Section \ref{procat}, we already defined pro-categories in the normal categorical setting. There, our main example was the category of profinite sets. In the following, we will generalise this notion to the context of (not necessarily small) $\infty$-categories. Recall that the notion of pro-categories is the dual concept of ind-categories and have a look at  Chapter 5 in \cite{Lurie2009} for a detailed discussion on ind-$\infty$-categories for small $\infty$-categories. Besides \cite{Lurie2009}, the content of this chapter is mainly based on \cite[Appendix A.8.1.]{Lurie2018a}.
\begin{definition}{\textbf{$\infty$-category of pro-objects}}\label{defpro}\\ 
Let $\cat$ be an accessible $\infty$-category which admits finite limits. We denote by $\mathrm{Fun^{lex}}(\cat, \Ani)$ the full subcategory of $\mathrm{Fun}(\cat, \Ani)$ spanned by functors $F \colon \cat \rightarrow \Ani$ which preserve finite limits (called left-exact functors) and are accessible. Then, we refer to the $\infty$-category defined by $$\Pro(\cat)\coloneqq\mathrm{Fun^{lex}}(\cat, \Ani)^{\mathrm{op}}$$ as \textit{$\infty$-category of pro-objects} of $\cat$.
\end{definition}
\begin{remark}{\textbf{Comment on accessibility}}\label{indcat}\\
By \cite[Definition 5.3.5.1. and Corollary 5.3.5.4.]{Lurie2009}, for a small $\infty$-category $\cat$ and some regular cardinal $\kappa$ the ind-$\infty$-category $\mathcal{I}\mathrm{nd}_{\kappa}(\cat^{\mathrm{op}})$ is given by all functors of $\Fun(\cat, \Ani)$ which are $\kappa$-filtered colimits of representable functors. Additionally, assume $\cat$ to be accessible and admit finite limits as in Definition \ref{defpro}. Then $\cat^{\mathrm{op}}$ admits finite colimits, so that $\mathcal{I}\mathrm{nd}(\cat^{\mathrm{op}})$ is the full subcategory of $\Fun(\cat, \Ani)$ spanned by those functors which preserve finite limits. In this setting, every functor $F\colon \cat\rightarrow \Ani$ is automatically accessible \cite[Corollary 5.4.3.6.]{Lurie2009} and one has a canonical isomorphism $\mathcal{I}\mathrm{nd}(\cat^{\mathrm{op}})^{\mathrm{op}}=\Pro(\cat)$ as stated in \cite[Remark A.8.1.2.]{Lurie2018a}. Thus, our definition of pro-$\infty$-categories constitutes a natural extension to a case where $\cat$ is not necessarily small and can be transferred in the same manner to the terminology of ind-$\infty$-categories. For technical reasons, it is convenient to add a hypothesis of accessibility \cite[Remark 7.1.6.2.]{Lurie2009}. 
\end{remark}
\begin{proposition}{\textbf{Yoneda embedding}} \label{Yoneda}\\
If $\cat$ is an accessible $\infty$-category which admits finite limits, then for every $C\in \cat$ the corepresentable functor $$\cat \rightarrow \Ani, \ X\mapsto \Hom_{\cat}(C, X)$$ is accessible and preserves finite limits. It follows that the opposite of the Yoneda embedding
$$\cat^{\mathrm{op}}\hookrightarrow \mathrm{Fun}(\cat, \Ani)$$ determines a fully faithful functor $$j_{\cat} \colon \cat \rightarrow \Pro(\cat),$$ which we will also refer to as the Yoneda embedding.
\end{proposition}
\begin{proof}
The corepresentable functor clearly preserves finite limits as it preserves in fact all limits. It preserves small filtered colimits, i.e. is accessible as we assumed $\cat$ to be accessible, as all corepresentable functors in functor categories are compact by an $\infty$-categorical translation of Remark \ref{reprecom} (2). More precisely, one has for some $C \in \cat$ by the $\infty$-categorical Yoneda lemma 
$$\Hom_{\Fun(\cat, \Ani)}(\Hom_{\cat}(C,\cdot),F)\cong F(C)$$
for all $F \in \Fun(\cat, \Ani)$. The evaluation functor 
$$\Fun(\cat,\Ani)\rightarrow \Ani, \ F \mapsto F(C)$$
commutes with all colimits as these can be calculated objectwise in a functor-$\infty$-category by \cite[Corollary 5.1.2.3.]{Lurie2009}. Therefore, 
$$F \mapsto \Hom_{\Fun(\cat, \Ani)}(\Hom_{\cat}(C,\cdot),F)$$
commutes with all colimits.
\end{proof}
\begin{remark}\label{pres}
The $\infty$-category $\mathrm{Fun^{lex}}(\cat, \Ani)$ of left-exact accessible functors is closed under small filtered colimits. It follows that the $\infty$-category $\Pro(\cat)$ admits small cofiltered limits. 
Moreover, $\cat$ can be seen as a subcategory of the cocompact pro-$\infty$-objects, as we have seen in the proof of Proposition \ref{Yoneda} that corepresentable functors are compact objects in $\Fun(\cat, \Ani)^{\mathrm{op}}$ .
\end{remark}
As a first example for a pro-$\infty$-category we will revisit the category of profinite sets - but now defined as an $\infty$-category. The following remark states how the usual categorical and the $\infty$-categorical definitions are correlated. 
\begin{remark}
We have the following nice fact: If $\cat$ is the nerve of an accessible ordinary category $I$, then $\Pro(\cat)$ is equivalent to the nerve of an ordinary category $J$ that is equivalent to the category $\mathrm{Pro}(I)$ of pro-objects of $I$ in the sense of ordinary category theory. Shortly, one has an equivalence of $\infty$-categories $N(\mathrm{Pro}(I))=\Pro(N(I))$, where $N(\mathcal{D})$ denotes the nerve of a category $\mathcal{D}$. This is the dual statement of \cite[Remark 5.3.5.6.]{Lurie2009}, again translated to accessible categories.
Hence, our notion of pro-categories in the usual categorical setting embeds in a natural way into the $\infty$-setting.
\end{remark}
\begin{example}{\textbf{Profinite sets}}\label{newpro}\\
By the remark above, we can identify the pro-$\infty$-category $\mathcal{P}\mathrm{ro}(\Set^{\mathrm{fin}})$ of the $\infty$-category $\Set^{\mathrm{fin}}$, i.e. the nerve of the ordinary category of finite sets, with the (nerve of the) pro-category $\PFin$. 
Therefore, we can transfer properties of $\PFin$ to our new setting.
\end{example}
\begin{proposition} \label{equiv}
Let $\cat$ be an accessible $\infty$-category which admits finite limits, let $\mathcal{D}$ be an $\infty$-category which admits small cofiltered limits, and let $\mathrm{Fun}'(\Pro(\cat),\mathrm{D})$ denote the full subcategory of $\mathrm{Fun}(\Pro(\cat), \mathcal{D})$ spanned by those functors which preserve small cofiltered limits. Then, composition with the Yoneda embedding restricts to an equivalence of $\infty$-categories $$\mathrm{Fun}'(\Pro(\cat), \mathcal{D}) \rightarrow \mathrm{Fun}(\cat, \mathcal{D}).$$
\end{proposition}
\begin{proof}
See \cite[Proposition A.8.1.6.]{Lurie2018a}.
\end{proof}
\begin{remark}
By translating \cite[Remark before 5.3.6.1.]{Lurie2009} to pro-$\infty$-categories, one can define them also by demanding the following properties which characterize $\Pro(\cat)$ up to equivalence:
\begin{itemize}
\item[1.] There exists a functor $j_{\cat}\colon \cat \rightarrow \Pro(\cat)$ (as in Proposition \ref{Yoneda}).
\item[2.] The $\infty$-category $\Pro(\cat)$ admits small cofiltered limits (as in Remark \ref{pres}).
\item[3.] For all $\infty$-categories $\mathcal{D}$ which admit small cofiltered limits the equivalence of $\infty$-categories in Proposition \ref{equiv} is satisfied.
\end{itemize} 
Hence, we will refer to Proposition \ref{equiv} as \textit{universal property} of $\Pro(\cat)$.
\end{remark}
\begin{lemma}
Let $F \colon \cat \rightarrow \Ani$ be an accessible functor which preserves finite limits. Then, $F$ regarded as an object in $\mathrm{Fun^{lex}}(\cat, \Ani)$ can be written as small filtered colimit $\colim_{\alpha \in A^{\mathrm{op}}}F_{\alpha}$, where each of the functors $F_{\alpha}\colon \cat \rightarrow \Ani$ is corepresentable by an object of $\cat$, i.e. $F$ is given as a functor $$Y\mapsto \colim_{\alpha \in A^{\mathrm{op}}}\Hom_{\cat}(X_{\alpha}, Y)$$ for some cofiltered diagram $X\colon A \rightarrow \cat$.
\end{lemma}
\begin{proof}
See proof of Proposition \ref{equiv} in \cite[Proposition A.8.1.6.]{Lurie2018a}.
\end{proof}
Consequently, by turning to the opposite category $\mathrm{Fun^{lex}}(\cat, \Ani)^{\mathrm{op}}=\Pro(\cat)$, the colimit in the lemma is indeed a limit in $\Pro(\cat)$ as here the direction of morphisms is reversed. This observation yields the following important corollary about the presentation of pro-objects.
\begin{corollary}{\textbf{Presentation of pro-objects}}\label{cofilli}\\
Every object in $\Pro(\cat)$ can be written as limit of a cofiltered diagram $\{F_{\alpha}\}_{\alpha \in A}$ where each $F_{\alpha}$ belongs to the essential image of the Yoneda embedding $j_{\cat}\colon \cat \rightarrow \Pro(\cat)$, i.e. a pro-object is given as a functor by
$$\mathrm{lim}_{\alpha \in A}  F_{\alpha}=\mathrm{lim}_{\alpha \in A}  j_{\cat}(X_{\alpha})=\mathrm{lim}_{\alpha \in A} \Hom_{\cat}(X_{\alpha}, \cdot)$$ in $\Pro(\cat)$.
\end{corollary}
This result can be stated as: The $\infty$-category $\Pro(\cat)$ consists of those functors which are "small cofiltered limits of corepresentables". We will abuse notation by identifying the diagram $\{F_{\alpha}\}_{\alpha \in A}$ with its limit in $\Pro(\cat)$ and denoting a pro-object as $\{X_{\alpha}\}_{\alpha \in A}$.
\begin{corollary}\label{collev}
If the $\infty$-category $\cat$ additionally admits finite coproducts, the same is true for its pro-$\infty$-category $\Pro(\cat)$ and these are given by the formula
$$\{X_{\alpha}\}_{\alpha \in A}\amalg \{Y_{\beta}\}_{\beta \in B} = \{X_{\alpha} \amalg Y_{\beta}\}_{(\alpha,\beta) \in A\times B}.$$ 
\end{corollary}
\begin{proof}
Consider pro-objects given by $\{X_{\alpha}\}_{\alpha \in A}$ and $\{Y_{\beta}\}_{\beta\in B}$. Viewing them as objects in $\mathrm{Fun^{lex}}(\cat, \Ani)$ we have for every object $I\in \cat$ the equivalence
$$\colim_{(\alpha,\beta)\in A\times B}\Hom_{\cat}(X_{\alpha}\amalg Y_{\beta}, I)\simeq (\colim_{\alpha \in A}\Hom_{\cat}(X_{\alpha}, I)\times (\colim_{\beta \in B}\Hom_{\cat}(Y_{\beta}, I))$$ as filtered colimits commute with finite products in $\Ani$ by Remark \ref{inftopo}. Passing to $\Pro(\cat)$ turns every colimit and limit outside the hom-space into its dual notion. This suffices to prove the claim.
\end{proof}
\begin{proposition}{\textbf{Morphisms of pro-objects}}\label{Homsets}\\
For pro-objects $X = \{X_{\alpha}\}_{\alpha \in A}$ and $Y = \{Y_{\beta}\}_{\beta \in B}$ we obtain the formula $$\Hom_{\Pro(\cat)}(X,Y)\simeq \mathrm{lim}_{\beta \in B} \colim_{\alpha \in A}\Hom_{\cat}(X_{\alpha},Y_{\beta}).$$
\end{proposition}
\begin{proof}
The result follows from the fully faithfulness of the Yoneda embedding and cocompactness of corepresentables in $\Pro(\cat)$ by Remark \ref{pres}. More precisely, by applying Corollary \ref{cofilli}, one can express the left hand side as
\begin{align*}
&\Hom_{\Pro(\cat)}(\mathrm{lim}_{\alpha \in A}j_{\cat}(X_{\alpha}),\mathrm{lim}_{\beta \in B}j_{\cat}(Y_{\beta}))\\
\simeq &\mathrm{lim}_{\beta \in B}\Hom_{\Pro(\cat)}(\colim_{\alpha\in A} j_{\cat}( X_{\alpha}),j_{\cat}(Y_{\beta}))\\
\simeq &\mathrm{lim}_{\beta \in B}\colim_{\alpha \in A}\Hom_{\Pro(\cat)}(j_{\cat}(X_{\alpha}), j_{\cat}(Y_{\beta}))\\
\simeq &\mathrm{lim}_{\beta \in B}\colim_{\alpha \in A}\Hom_{\cat}(X_{\alpha},Y_{\beta}),
\end{align*}
where we used cocompactness of corepresentables in the first identification and fully faithfulness of the Yoneda embedding in the third. 
\end{proof}
\begin{corollary}\label{funcacc}
Let $f \colon \cat \rightarrow \mathcal{D}$ be a functor between accessible $\infty$-categories which admit finite limits. Then the composite functor $$\cat
\xrightarrow{f} \mathcal{D} \xrightarrow{j_{\mathcal{D}}} \Pro(\mathcal{D})$$ induces a unique functor $$\Pro(f)\colon \Pro(\cat) \rightarrow \Pro(\mathcal{D}),$$ which commutes with small cofiltered limits and whose restriction to $\cat$ coincides with $j_{\mathcal{D}}\circ f$. If $f$ is fully faithful, then $\Pro(f)$ is fully faithful as well. 
\end{corollary}
\begin{proof}
By Remark \ref{pres}, the $\infty$-category $\Pro(\mathcal{D})$ admits small cofiltered limits. Therefore, we can apply Proposition \ref{equiv} to the composite functor
\begin{align}\label{composite}
\cat\xrightarrow{f} \mathcal{D} \xrightarrow{j_{\mathcal{D}}} \Pro(\mathcal{D})
\end{align} to obtain a functor
$$\Pro(f)\colon \Pro(\cat) \rightarrow \Pro(\mathcal{D})$$ preserving small cofiltered limits and satisfying
\begin{align}\label{commu}
\Pro(f)\circ j_{\cat} = j_{\mathcal{D}} \circ f.
\end{align}
If $f$ is fully faithful, one can conclude the fully faithfulness of $\Pro(f)$ by Proposition \ref{Homsets}. Indeed, for $X=\{X_{\alpha}\}_{\alpha \in A},Y=\{Y_{\beta}\}_{\beta \in B} \in \Pro(\cat)$ one has
\begin{align*}
\Hom_{\Pro(\cat)}(X,Y)&\simeq \mathrm{lim}_{\beta \in B} \colim_{\alpha \in A}\Hom_{\cat}(X_{\alpha},Y_{\beta})\\
&\simeq \mathrm{lim}_{\beta \in B} \colim_{\alpha \in A}\Hom_{\mathcal{D}}(f(X_{\alpha}),f(Y_{\beta}))\\
&\simeq \Hom_{\Pro(\mathcal{D})}(\mathrm{lim}_{\alpha \in A}\ j_{\mathcal{D}}(f(X_{\alpha})),\mathrm{lim}_{\beta \in B}\ j_{\mathcal{D}} (f(Y_{\beta})))\\  
&\simeq \Hom_{\Pro(\mathcal{D})}(\Pro(f)(X),\Pro(f)(Y)),
\end{align*}
where we used (\ref{commu}) and that $\Pro(f)$ preserves cofiltered limits for the last identification.
\end{proof}
\begin{proposition}{\textbf{Pro-existent left adjoint}}\label{proext}\\
Let $\cat$ be an accessible $\infty$-category that admits finite limits and $\mathcal{D}$ be an accessible $\infty$-category that admits finite limits and small cofiltered limits. Then, by Proposition \ref{equiv} a left exact functor $$u\colon\cat\rightarrow \mathcal{D} $$ can be extended to a unique functor $$U:\Pro(\cat)\rightarrow \mathcal{D}$$ that preserves small cofiltered limits. In the other direction, one can consider the functor
$$L\coloneqq u^{*}\circ j_{\mathcal{D}}\colon \mathcal{D} \rightarrow \Pro(\mathcal{D})\rightarrow \Pro(\cat)$$ obtained by composition of the Yoneda embedding with the restriction along $u$. The functor $L$ carries an object $d\in \mathcal{D}$ to the assignment $$c\mapsto \Hom_{\mathcal{D}}(d, u(c))$$ and is left adjoint to $U$. Hence, the functor is called pro-existent left adjoint to $u$.
\end{proposition}
\begin{proof}
This result is stated without proof in \cite[Subsection 0.11.9.]{Barwick2018}. We will give a proof here.\\
Clearly, the functor $u^{*}\circ j_{\mathcal{D}}\colon \mathcal{D} \rightarrow \Pro(\mathcal{D})\rightarrow \Fun(\cat, \Ani)^{\mathrm{op}}$ is given by the assignment stated above. As $u$ in the second argument of the hom-space preserves finite limits by assumption, the same is true for $$c\mapsto \Hom_{\mathcal{D}}(d, u(c)).$$ Thus, $u^{*}\circ j_{\mathcal{D}}$ factors through $\Pro(\cat)$. The left adjointness can be easily shown by using the presentation of pro-$\infty$-objects as cofiltered limits and the fact that $U$ preserves these limits:
One has for some $d \in \mathcal{D}$ and $C=\{C_\alpha\}_{\alpha \in A} \in \Pro(\cat)$
\begin{align*}
\Hom_{\mathcal{D}}(d, U(C))&=\lim_{\alpha \in A} \Hom_{\mathcal{D}}(d, U(C_{\alpha}))\\
&=\lim_{\alpha \in A} \Hom_{\mathcal{D}}(d, u(C_{\alpha}))\\
&=\lim_{\alpha \in A} \Hom_{\Pro(\cat)}(\Hom_{\mathcal{D}}(d, u(\cdot)), \Hom_{\cat}(C_{\alpha}, \cdot))\\
&=\Hom_{\Pro(\cat)}(\Hom_{\mathcal{D}}(d, u(\cdot)), \lim_{\alpha \in A}\Hom_{\cat}(C_{\alpha}, \cdot))\\
&=\Hom_{\Pro(\cat)}(\Hom_{\mathcal{D}}(d, u(\cdot)), C),
\end{align*}
where the third equivalence is just the $\infty$-categorical Yoneda lemma.
\end{proof}
\section{The \texorpdfstring{$\infty$}{infinity}-category of pro-anima}
\begin{definition}{\textbf{Pro-anima}} \\
A \textit{pro-anima} is a pro-object of the $\infty$-category of anima, i.e. an
accessible left-exact functor from the $\infty$-category $\Ani$ to itself. We will refer to the $\infty$-category $\Pro(\Ani) \subset \mathrm{Fun}(\Ani, \Ani)^{\mathrm{op}}$
as the \textit{$\infty$-category of pro-anima}.
\end{definition}
\begin{remark}
By Proposition \ref{Yoneda}, for any anima $A$ the associated pro-anima is given as the functor $$\Ani \rightarrow \Ani, \ X \mapsto \Hom_{\Ani}(A, X)$$ represented by $A$.
\end{remark}
\begin{proposition}\label{proaniprofin}
Every profinite set $S=\{S_i\}_{i \in I}$ can be identified with a pro-anima given by
\begin{align}\label{eq2}
\Ani \rightarrow \Ani, \ A\mapsto\mathrm{lim}_{i \in I}\Hom_{\Ani}(v(S_{i}),A)
\end{align}
where $v\colon \Set \hookrightarrow \Ani$ is the embedding. By this, one can view $\Pro(\Set^{\mathrm{fin}})$ as a full subcategory of $\Pro(\Ani)$ and the embedding preserves cofiltered limits.
\end{proposition}
\begin{proof}
Every set can be regarded as an anima by identifying it with a discrete anima. Hence, there exists a pro-anima associated to every profinite set. In particular, one has a fully faithful functor $v|_{\Set^{\mathrm{fin}}}\colon\Set^{\mathrm{fin}}\hookrightarrow \Ani$ preserving finite limits as right adjoint by Proposition \ref{truncationprop}. This functor induces a fully faithful functor 
$$\Pro(v)\colon \Profin \rightarrow \Pro(\Ani)$$ by Corollary \ref{funcacc}. As $\Pro(v)$ preserves small cofiltered limits, the pro-anima corresponding to a profinite set $S=\mathrm{lim}_{i \in I} S_i$ written as small cofiltered limit of finite (discrete) sets is given as a functor 
\begin{align}
\Ani \rightarrow \Ani, \ A\mapsto\mathrm{lim}_{i \in I}\Hom_{\Ani}(v(S_{i}),A)
\end{align} in $\Pro(\Ani)$ by Corollary \ref{cofilli}.\\
\end{proof}
\begin{remark}\label{profinproobj}
Let $\cat$ be an $\infty$-category such that $\Pro(\cat)$ is defined and there exists a fully faithful functor $v_{\cat}\colon\Set^{\mathrm{fin}} \rightarrow \cat$ preserving finite limits. Then the proposition above can be stated more generally:\\
Every profinite set $S=\{S_i\}_{i \in I}$ can be identified with a pro-object given by
\begin{align*}
\cat \rightarrow \Ani, \ A\mapsto\mathrm{lim}_{i \in I}\Hom_{\cat}(v_{\cat}(S_{i}),A)
\end{align*}
in $\Pro(\cat)$. By this, one can view $\Pro(\Set^{\mathrm{fin}})$ as a full subcategory of $\Pro(\cat)$ and the embedding preserves cofiltered limits.
\end{remark}

\chapter{Condensed anima}\label{condani}
In the following chapter, we want to expand the notion of condensed sets according to our new setting of $\infty$-categories. There is a general notion of condensed objects of an $\infty$-category.
The natural expansion of condensed sets is given by the $\infty$-category of condensed anima. In Theorem \ref{thmAniCond} we will explore how this is related to what we have done in Chapter \ref{animation}.\\
The content of this chapter is highly inspired by the work of Peter Haine and Clark Barwick on \textit{Pyknotic spaces}. Corresponding results can be found in \cite{Barwick2018, Barwick2019} and were also presented by Peter Haine during his talks at MSRI in March 2020 \cite{Haine2020a, Haine2020}. Condensed anima were also part of \href{https://www.youtube.com/channel/UCuDrbQmLiT0jkSlMKPN-fsQ/videos}{the Sessions 9\&10} held by Peter Scholze at the \textit{Masterclass in Condensed Mathematics} in November 2020 \cite{Clausen}.
\section{Condensed anima as \texorpdfstring{$\infty$}{infinity}-sheaves}\label{condob}
\begin{definition}{\textbf{Presheaves of anima}}\label{pre}\\
Let $\cat$ be an $\infty$-category which admits finite limits.
We define a presheaf of anima to be a functor
$$F \colon \cat^{\mathrm{op}} \rightarrow \Ani.$$ We will denote the $\infty$-category $\Fun(\cat^{\mathrm{op}}, \Ani)$ by $\mathcal{PS}\mathrm{h}(\cat)$ and refer to it as the \textit{$\infty$-category of presheaves (of anima) on $\cat$}.
\end{definition}
\begin{proposition}\label{extpresh} 
Let $\cat$ be an accessible $\infty$-category which admits finite limits.
There is an equivalence of $\infty$-categories 
\begin{align*}
\mathcal{PS}\mathrm{h}'(\Pro(\cat))\rightarrow \mathcal{PS}\mathrm{h}(\cat),
\end{align*}
where the left $\infty$-category is the full-$\infty$-category of presheaves preserving cofiltered limits of $\Pro(\cat)$. The equivalence is given by composition with the opposite functor $\iota_{\cat}\colon \cat^{\mathrm{op}}\rightarrow \Pro(\cat)^{\mathrm{op}}$ of the Yoneda embedding $j_{\cat}$. Its inverse is given by forming left Kan extensions along $\iota_{\cat}$ .
\end{proposition}
\begin{proof}
By \cite[Proposition 5.3.5.10.]{Lurie2009}, for a small $\infty$-category $\cat$ and a regular cardinal $\kappa$ the composition with the Yoneda embedding 
$\cat^{\mathrm{op}} \rightarrow \mathcal{I}\mathrm{nd}_{\kappa}(\cat^{\mathrm{op}})$ induces an equivalence of $\infty$-categories
$$\Fun_{\kappa}'(\mathcal{I}\mathrm{nd}_{\kappa}(\cat^{\mathrm{op}}), \Ani)\rightarrow \Fun(\cat^{\mathrm{op}}, \Ani),$$
where the left hand side denotes the $\infty$-category of all functors from $\mathcal{I}\mathrm{nd}_{\kappa}(\cat^{\mathrm{op}})$ to $\Ani$ preserving $\kappa$-filtered colimits. By Remark \ref{indcat}, there is a natural extension of this result to accessible categories similar to Proposition \ref{equiv}. As the ind-$\infty$-category $\mathcal{I}\mathrm{nd}(\cat^{\mathrm{op}})$ coincides with $\Pro(\cat)^{\mathrm{op}}$, the equivalence follows. By \cite[Lemma 5.3.5.8.]{Lurie2009}, the inverse of the equivalence is given by forming the left Kan extensions along $\iota_{\cat}\colon \cat^{\mathrm{op}} \rightarrow \mathrm{Pro}(\cat)^{\mathrm{op}}$.
\end{proof}
The following lemma presents how the inverse of the equivalence can be constructed explicitly.
\begin{lemma}\label{constrext}
Let $\cat$ be given as in Proposition \ref{extpresh} and $F \in \mathcal{PS}\mathrm{h}(\cat)$. The by Proposition \ref{extpresh} to $F$ associated presheaf in $\mathcal{PS}\mathrm{h}'(\Pro(\cat))$ is defined by the formula
$$\widetilde{F}\colon X=\{X_i\}_{i \in I} \rightarrow \colim_{i \in I^{\mathrm{op}}}F(X_i),$$ where the colimit is filtered.
\end{lemma}
\begin{proof}
First one has to clarify that the given construction defines a contravariant functor into the category of anima. We have $$\Hom_{\mathrm{Pro}(\cat)}(X, Y) \coloneqq \mathrm{lim}_{i \in I} \colim_{j \in J} \Hom_{\cat}(X_i, Y_j)$$ for morphisms of pro-$\infty$-objects and for every morphism $X_i\rightarrow Y_j$ of objects in $\cat$ we obtain a morphism $F(Y_j)\rightarrow F(X_i)$ of anima. This induces a map
 $$\Hom_{\mathrm{Pro}(\cat)}(X, Y) \rightarrow \colim_{i \in I} \mathrm{lim}_{j \in J} \Hom_{\Ani}(F(Y_j), F(X_i)).$$ As one can pull the limit into the first argument of the hom-space and there is a natural map of hom-spaces
$$ \colim_{i \in I} \Hom_{\Ani}(\colim_{j \in J} F(Y_j), F(X_i)) \rightarrow   \Hom_{\Ani}(\colim_{j \in J}F(Y_j), \colim_{i \in I}F(X_i)),$$ every morphism $f\colon X\rightarrow Y$ of objects in $\Pro(\cat)$ can be associated with a morphism $\widetilde{F}(f)\colon F(Y) \rightarrow F(X)$. Thus $\widetilde{F}$ defines a contravariant functor.
It preserves cofiltered limits of $\Pro(\cat)$ by construction. Its composition with the opposite of the Yoneda embedding $j_{\cat}$ obviously coincides with $F$. Thus it is contained in $\mathcal{PS}\mathrm{h}'(\Pro(\cat))$ and presents a preimage of $F$ under composition with the Yoneda embedding. But as one has an equivalence of $\infty$-categories, the preimage is uniquely determined up to equivalence.
\end{proof}
To define the notion of sheaves of anima, we will restrict ourselves to sites in the usual categorical setting as this is all we need for our purposes. For a discussion on Grothendieck topologies and sheaves in the $\infty$-setting see \cite[Section 6.2.2.]{Lurie2009}.
\begin{definition}{\textbf{Sheaves of anima}}\\
Let $\cat$ be a site with finite limits in the usual categorical setting  viewed as $\infty$-category. A \textit{sheaf (of anima)} is a presheaf $F$ of anima such that for all coverings $(f_i\colon X_i \rightarrow X)_i$ the map of anima
\begin{align*}
F(X) \rightarrow \lim (\prod_i F(X_i) \rightrightarrows \prod_{i,j} F(X_i \times_X X_j) \substack{\rightarrow\\[-1em] \rightarrow \\[-1em] \rightarrow} \prod_{i,j,k} F(X_i \times_X X_j \times_X X_k) \dots),
\end{align*}
where the limit is taken over an infinite diagram, is an equivalence.
We denote the $\infty$-category of all sheaves of anima on $\cat$ by $\mathcal{S}\mathrm{h}(\cat)$.
\end{definition}
\begin{definition}{\textbf{Condensed anima}}\label{condanidef}\\
The $\infty$-category $\CondAni$ of \textit{condensed anima} is defined as the $\infty$-category of functors
$$\extdis^{\mathrm{op}} \rightarrow \Ani$$
in $\mathcal{PS}\mathrm{h}(\extdis)$ taking finite disjoint unions to finite products. As for condensed sets this is obtained by restricting sheaves on the site of profinite sets to $\extdis$.
\end{definition}
\begin{remark}{\textbf{Condensed objects}}\\
More generally, one can define for some $\infty$-category $\mathcal{D}$ admitting certain limits and colimits the category of \textit{condensed objects of $\mathcal{D}$} denoted by $\mathcal{C}\mathrm{ond}(\mathcal{D})$. Its objects are, according to the definitions above, given by contravariant functors from extremally disconnected profinite sets to $\mathcal{D}$ that take finite disjoint unions to finite products. As for condensed sets, this definition has some cardinality issues. See \cite[Definition 11.7]{Scholze2019a} for a precise definition avoiding these issues.
\end{remark}
\begin{lemma}\label{exttoconani}
Every presheaf $F\colon (\Set^{\mathrm{fin}})^{\mathrm{op}}\rightarrow \Ani$ sending finite disjoint unions to finite products can be extended to a unique condensed anima preserving cofiltered limits.
\end{lemma}
\begin{proof}
The $\infty$-category of finite sets satisfies the conditions of Proposition \ref{extpresh}. Thus, we obtain a unique presheaf $\widetilde{F} \in \mathcal{PS}\mathrm{h}(\Pro(\cat))$ by the formula
$$X=\{X_i\}_{i \in I} \rightarrow \colim_{i \in I^{\mathrm{op}}}F(X_i)$$ preserving cofiltered limits. One has $X \amalg Y =\{X_i \amalg Y_j\}_{(i,j) \in I \times J}$ for a disjoint union of profinite sets by Corollary \ref{collev}. Combining this with the assumption of $F$ sending finite disjoint unions to finite products and the fact that filtered colimits commute with finite products in $\Ani$ by Remark \ref{inftopo}, one can see that $\widetilde{F}$ sends finite disjoint unions to finite products as well. Hence, its restriction to $\extdis^{\mathrm{op}}$ defines a condensed anima.
\end{proof}
\section{Discrete condensed anima}\label{condaninew}
In Definition \ref{anicond} we already defined $\AniCond$ as the full-$\infty$-subcategory of
$$\mathrm{Fun}(\extdis^{\mathrm{op}}, \Ani),$$ generated under sifted colimits by the Yoneda image of $\extdis$. The results of this section give rise to a more precise description of these functors. Furthermore, Theorem \ref{thmAniCond} shows how $\AniCond$ unites topological and homotopical structures at the same time.
\begin{lemma}\label{constinC}
Let $\cat$ be an $\infty$-category and $v_{\cat}\colon \Set^{\mathrm{fin}} \rightarrow \cat$ a fully faithful functor. To every object $A\in \cat$, one can assign a presheaf of anima $$A^{\mathrm{disc}}\colon\Pro(\Set^{\mathrm{fin}})^{\mathrm{op}} \rightarrow \Ani$$ preserving cofiltered limits in $\Profin$ and given by the formula
\begin{align}\label{form}
X=\{X_i\}_{i \in I} \mapsto \colim_{i \in I^{\mathrm{op}}}\Hom_{\cat}(v_{\cat}(X_i),A),
\end{align}
where the colimit is filtered.
\end{lemma}
\begin{proof}
The presheaf is the extension of the presheaf of anima $$\Hom_{\cat}(v_{\cat}(\cdot),A)\colon (\Set^{\mathrm{fin}})^{\mathrm{op}}\rightarrow \Ani$$ to $\Profin$ as given in Lemma \ref{constrext}. With regards to Proposition \ref{extpresh}, it preserves cofiltered limits in $\Profin$.
\end{proof}
\begin{proposition}\label{constsheaf} 
If additionally $\cat$ admits finite coproducts and $v_{\cat}$ preserves this kind of coproducts, the restriction of $A^{\mathrm{disc}}$ to the subcategory  $\extdis^{\mathrm{op}}$ defines a condensed anima.
\end{proposition}
\begin{proof}
As $v_{\cat}$ preserves finite coproducts by assumption and one has the functorial equivalence $$\Hom_{\cat}(\colim_{i \in I} B_i,A) = \mathrm{lim}_{i \in I} \Hom_{\cat}(B_i,A)$$ of hom-spaces, the presheaf of anima $\Hom_{\cat}(v_{\cat}(\cdot),A)\colon (\Set^{\mathrm{fin}})^{\mathrm{op}}\rightarrow \Ani$ sends finite disjoint unions to finite products. By restriction we just obtain the construction of Lemma \ref{exttoconani}.
\end{proof}
We will use the results above to define a condensed anima associated to an anima. In Proposition \ref{proright} we will obtain by the same formula a condensed anima associated to a pro-anima.
\begin{definition}{\textbf{Discrete condensed anima}}\label{discr}\\
Let $\cat\coloneqq \Ani$ be the category of animated sets. For an anima $A$ we call the condensed anima obtained by Proposition \ref{constsheaf} the \textit{discrete condensed anima} attached to $A$ and denote it by $\underline{A}^{\mathrm{disc}}$. The functor
$$\underline{(\cdot)}^{\mathrm{disc}}\colon \Ani \rightarrow \CondAni, \ A\mapsto \underline{A}^{\mathrm{disc}}$$ is called \textit{discrete functor} or \textit{constant sheaf functor}.
\end{definition}
We have another functor involving the categories $\Ani$ and $\CondAni$ which is nicely connected to the discrete functor as we will see.
\begin{definition}{\textbf{Global sections functor}}\label{global}\\
The \textit{global sections functor} on condensed anima is defined as the functor
\begin{align}
\Gamma\colon \CondAni \rightarrow \Ani, \ F \mapsto F(\ast)
\end{align}
sending a condensed anima to its evaluation on the final object $\ast$ of $\Profin$.\\ (See Definition \ref{finaldef} for a precise definition of final objects in the $\infty$-setting).
\end{definition}
Recall that every discrete condensed anima is by definition and Proposition \ref{extpresh} (a restriction of) a left Kan extension along $\iota_{\SFin}\colon (\SFin)^{\mathrm{op}} \rightarrow \Pro(\Set^{\mathrm{fin}})^{\mathrm{op}}$. The following proposition can be proven in an abstract way by applying the fact that forming the left Kan extension along some morphism $G$ is precisely the left adjoint $\infty$-functor of precomposition with this $G$.  Note the subsequent lemma for a reference of this fact. Nevertheless, in the proof of the proposition we want to explicitly define how the claimed adjunction of the global sections functor and the discrete functor can be obtained.
\begin{proposition}\label{leftadjoint}
The global sections functor provides a right adjoint to the discrete functor, i.e. one has a functorial equivalence of hom-spaces
\begin{align}\label{adjun}
\Hom_{\CondAni}(\underline{A}^{\mathrm{disc}}, F)\simeq \Hom_{\Ani}(A, F(\ast))
\end{align}
for all $A\in \Ani$ and $F\in \CondAni$.
\end{proposition}
\begin{proof}
First, we want to show the equivalence
\begin{align}\label{first}
\Hom_{\CondAni}(\underline{A}^{\mathrm{disc}}, F)\simeq \Hom_{\mathcal{PS}\mathrm{h}(\SFin)}(\overline{A}, F|_{\SFin})
\end{align}
where $\overline{A}\coloneqq \Hom_{\Ani}(v_{\SFin}(\cdot),A)$ is the representable functor of $A\in \Ani$ restricted to $\SFin$.\\
In order to do so, we want to show that for every $F\in \CondAni$, any morphism $h\colon \overline{A} \rightarrow F|_{\SFin}$ corresponds to a unique morphism $\underline{A}^{\mathrm{disc}} \rightarrow F$ of condensed anima, s.t. restriction to $\SFin$ yields the morphism we started with. Note that $\underline{A}^{\mathrm{disc}}|_{\SFin}$ is just given by the representable functor $\overline{A}=\Hom_{\Ani}(v_{\SFin}(\cdot),A)$ on $\SFin$. For every extremally disconnected profinite set $X=\{X_i\}_{i \in I}$, one has maps $X \rightarrow X_i$ and thus maps $F(X_i) \rightarrow F(X)$ and $\overline{A}(X_i)\rightarrow F(X_i)$  being compatible with all arrows in the diagram $\{X_i\}_{i \in I}$. Hence, one obtains for every such $X$ a unique map $$\colim_{i \in I^{\mathrm{op}}} \overline{A}(X_i) \rightarrow F(X)$$ by the universal property of colimits. The resulting map $N \colon \underline{A}^{\mathrm{disc}} \rightarrow F$ is a natural transformation extending $h$ in the required way. As every other morphism $\underline{A}^{\mathrm{disc}} \rightarrow F$ extending $h$ has to satisfy the same conditions as we used for our construction, it has to be equivalent to $N$. It remains to show the equivalence
\begin{align}\label{secon}
\Hom_{\mathcal{PS}\mathrm{h}(\SFin)}(\overline{A}, F|_{\SFin})\simeq \Hom_{\Ani}(A, F(\ast)).
\end{align}
Every finite set can be expressed as a finite disjoint union of one point sets. As $\overline{A}$ and $F|_{\SFin}$ define presheaves on $\SFin$ which send finite disjoint unions to finite products, every morphism $\overline{A} \rightarrow F|_{\SFin}$ is fully determined by its evaluation on the finite set $\{\ast\}$ which is the final object $\ast$ in $\Profin$. Hence, the morphism $\overline{A} \rightarrow F|_{\SFin}$ can be uniquely identified with a morphism $\overline{A}(\ast) \rightarrow F(\ast)$ of anima where $\overline{A}(\ast)=\Hom_{\Ani}(\{\ast\},A)=A$. This proves (\ref{secon}) which combined with (\ref{first}) gives the adjunction (\ref{adjun}).
\end{proof}
As mentioned, the adjunction (\ref{first}) can be formulated more generally for left Kan extensions.
\begin{lemma}
Let $\cat$, $\widetilde{\cat}$ and $\mathcal{D}$ be $\infty$-categories.
For every left Kan extension $LS\colon \widetilde{\cat} \rightarrow \mathcal{D}$ of a morphism $S\colon \cat \rightarrow \mathcal{D}$ along some $G\colon \cat \rightarrow \widetilde{\cat}$, one has the equivalence
\begin{align*}
\Hom_{\Fun(\widetilde{\cat}, \mathcal{D})}(LS, F)\simeq \Hom_{\Fun(\cat, \mathcal{D})}(S, F\circ G),
\end{align*}
natural in $F \in \Fun(\widetilde{\cat}, \mathcal{D})$. 
\end{lemma}
\begin{proof}
This follows by \cite[Proposition 4.3.3.7.]{Lurie2009}.
\end{proof}
\section{Topological spaces as condensed anima}\label{topconani}
Recall the fully faithful embedding $$\mathcal{CG} \hookrightarrow \Cond,\ X \mapsto \Hom_{\Top}(\cdot, X)$$ as mentioned in Proposition \ref{equivcomphaus}. Using this and part (c) of the following theorem, we can view compactly generated topological spaces as condensed anima. Moreover, the theorem summarizes important relations of the $\infty$-categories of anima, condensed anima and animated condensed sets.
\begin{theorem}\label{thmAniCond}
With the definitions above, the following are true:
\begin{enumerate}
\item[(a)] There is a natural equivalence of the $\infty$-categories 
$$\mathrm{Cond(}\mathcal{A}\mathrm{ni(}\Set))\cong \AniCond$$
of condensed anima and animated condensed sets.
\item[(b)] The $\infty$-category $\Ani$ embeds fully faithfully into $\AniCond$ by the discrete functor $\underline{(\cdot)}^{\mathrm{disc}}\colon \Ani \hookrightarrow \CondAni$ and
\item[(c)] the fully faithful embedding of the category $\Cond$ into $\AniCond$ of Proposition \ref{truncationprop2} is given by composition of a condensed set with the fully faithful embedding $v\colon\Set \hookrightarrow \Ani$. 
\end{enumerate}
\end{theorem}
\begin{proof}
\begin{enumerate}
\item[(a)] This is Lemma 11.8 in \cite{Scholze2019a}.
\item[(b)] We can conclude the statement directly as a corollary of Proposition \ref{leftadjoint} by choosing $F\coloneqq \underline{B}^{\mathrm{disc}}$ for some anima $B \in \Ani$ and noting that $\underline{B}^{\mathrm{disc}}(\ast)=B$.

\item[(c)] First of all, we note that for every $F\in \Cond$ the composition $v\circ F$ is a condensed anima by $(a)$. Indeed, $F$ sends finite disjoint unions to finite products by definition and the embedding $v$ preserves limits as it is a right adjoint by Proposition \ref{truncationprop}. By Proposition \ref{truncationprop2}, we already know that there is a fully faithful embedding $\Cond \hookrightarrow \AniCond$ that identifies $\Cond$ with the $0$-truncated objects and is right adjoint to the composition with $\pi_0\colon \Ani \rightarrow \Set$. It is enough to prove that the composition with the embedding $v\colon \Set \hookrightarrow \Ani$ is also right adjoint to the composition with $\pi_0$. Then, the fully faithfulness follows by uniqueness of adjoint functors. By the adjointness in Proposition \ref{truncationprop} one has
$$\Hom_{\Ani}(v\circ F(X), G(X)) \cong \Hom_{\Set}(F(X), \pi_0\circ G(X))$$ for all $F\in \Cond$, $G \in \AniCond$ and $X\in \extdis$. This immediately implies
$$\Hom_{\CondAni}(v\circ F, G) \cong \Hom_{\Cond}(F, \pi_0\circ G).$$
\end{enumerate}
\end{proof}
\begin{definition}
For a compactly generated topological space $X$, we call the condensed anima $\Hom_{\Top}(\cdot,X)\colon \extdis^{\mathrm{op}} \rightarrow \Ani$ obtained by the fully faithful embedding of Theorem \ref{thmAniCond} (c) the \textit{condensed anima corresponding to X} or \textit{represented by X}.
\end{definition}
\begin{proposition}\label{discfin}
Regarding $\SFin$ as a sub-$\infty$-category of $\Ani$ via the fully faithful embedding $v \colon \Set \hookrightarrow \Ani$ one can deduce the following:
\begin{enumerate}
\item For any finite set $S$, the discrete condensed anima $\underline{S}^{\mathrm{disc}}$ is the condensed anima $\Hom_{\Top}(\cdot, S)$ represented by $S$ considered as a finite discrete space.
\item  If $\{X_i\}_{i \in I}$ is a cofiltered system of finite sets, the limit $\mathrm{lim}_{i \in I} \underline{X_i}^{\mathrm{disc}}$ is the condensed anima $\Hom_{\Top}(\cdot, \mathrm{lim}_{i \in I} X_i)$ represented by the profinite set $\mathrm{lim}_{i \in I} X_i$.
\end{enumerate}
\end{proposition}
\begin{proof}
\begin{enumerate}
\item Let $S$ be a finite set and $v\colon\Set \hookrightarrow \Ani$ the embedding of sets to animated sets. Then, the discrete condensed anima associated to $S$ is given by the formula
$$X=\{X_i\}_{i \in I}\mapsto \colim_{i \in I^{\mathrm{op}}}\Hom_{\Ani}(v(X_i),v(S))$$
with $X$ extremally disconnected profinite set. As $v$ and the Yoneda embedding $\Set^{\mathrm{fin}}\hookrightarrow \Profin$ of finite sets into profinite sets are fully faithful, we can deduce
\begin{align*}
\colim_{i \in I^{\mathrm{op}}}\Hom_{\Ani}(v(X_i),v(S))
&=\colim_{i \in I^{\mathrm{op}}}\Hom_{\Set}(X_i, S)\\
&=\colim_{i \in I^{\mathrm{op}}}\Hom_{\Profin}(X_i, S),
\end{align*}
using that $S$ and $X_i$ are finite sets for all $i \in I$.
Moreover, every finite set is cocompact in $\Profin$ 
as stated in Remark \ref{pres}. Therefore, the filtered colimit can be computed as cofiltered limit inside the hom-space.
It follows that the discrete condensed anima can be stated by the formula
$$X=\{X_i\}_{i \in I}\mapsto \Hom_{\Profin}(\mathrm{lim}_{i \in I} X_i, S)=\Hom_{\Top}(X, S),$$
as we can view $\Profin$ as a full sub-$\infty$-category of (the nerve of) $\Top$ (since this is true in the ordinary categorical setting). This is what we wanted to prove.
\item Let $\{X_i\}_{i \in I}$ be a cofiltered system of finite sets. As the limit $\mathrm{lim}_{i \in I} \underline{X_i}^{\mathrm{disc}}$ can be written as
$$\mathrm{lim}_{i \in I} \Hom_{\Top}(\cdot, X_i)=\Hom_{\Top}(\cdot, \mathrm{lim}_{i \in I}X_i),$$
the statement directly follows from (1.). 
\end{enumerate}
\end{proof}
\begin{remark}
By Theorem \ref{thmAniCond} one obtains two natural functors from CW-spaces to condensed anima. On the one hand, there is a fully faithful
functor via the full subcategory of condensed sets. On the other hand, there is a non-faithful functor factoring over the full-$\infty$-subcategory of anima as $\CW \rightarrow \Ani$ is non-faithful. By Proposition \ref{discfin}, the images under the two functors coincide if 
the CW-space is a discrete finite set.  
\end{remark}

\chapter{Homotopy theory for condensed sets}\label{homotopy}
The aim of this chapter is to develop an option of doing homotopy theory with condensed sets in a meaningful way.\\
In his post \cite{Clausena} at April 5, 2020 in the \textit{$n$-Category caf\'{e}},  Clausen claims that one can regard $\AniCond\simeq \CondAni$ as an "$\infty$-topos modulo cardinality issues". If we believe this statement for a moment without demanding a proof, we can think about categorical homotopy groups as they are defined for $\infty$-topoi \cite[Definition 6.5.1.1.]{Lurie2009}. Then by \cite[Proposition 6.5.1.7.]{Lurie2009}, for all $0$-truncated objects the non-zero homotopy groups vanish. Hence, by considering $\Cond$ as the full sub-$\infty$-category of $0$-truncated objects of $\AniCond$ by Proposition \ref{truncationprop2}, this approach of defining a kind of homotopy groups for condensed sets is not very meaningful.
\newline
Another, more promising, approach is to investigate to what extent the embedding $\Ani \hookrightarrow \CondAni$ of Theorem \ref{thmAniCond} has a left adjoint. By Lemma 11.9 in \cite{Scholze2019a}, there exists at least a left adjoint if we restrict the functor to the full subcategory of CW-spaces regarded as condensed anima. The functor is given by sending the condensed anima corresponding to a CW-space to the singular simplicial complex of the CW-space as we will see in Proposition \ref{CW}. In fact, by Section \ref{example} we cannot hope that this partially left adjoint can be extended to the whole category $\CondAni$. Nevertheless, the construction of pro-$\infty$-categories allows us to define a left adjoint if we replace $\Ani$ by its pro-$\infty$-category $\Pro(\Ani)$. Then, as we will see,
one can define homotopy pro-groups.
\section{Pointed \texorpdfstring{$\infty$}{infinity}-categories}
In order to define homotopy groups, we have to define a notion of pointed objects of an $\infty$-category. We follow the definition in \cite[Section 7.2.2]{Lurie2009}.
\begin{definition}{\textbf{The $\infty$-category of pointed objects}}\label{finaldef}\\
Let $\cat$ be an $\infty$-category.
\begin{itemize}
\item[1.] We call an object $X\in \cat$ final if $\Hom_{\cat}(Y,X)$ is a contractible simplicial set for all $Y \in \cat$, i.e. it is weakly equivalent to the point. We denote a final object by $\ast_{\cat}$.
\item[2.] A pointed object of $\cat$ is a morphism $y\colon \ast_{\cat} \rightarrow Y$ in $\cat$. The full subcategory of $\cat_1=\Fun(\triangle^1,\cat)$ spanned by pointed objects is denoted by $\cat_{*}$ and called pointed $\infty$-category of $\cat$. 
\end{itemize}
\end{definition}
\begin{remark}\label{element}
For simplicity, we write $x$ for an object of $\cat_{*}$ meaning the morphism $x \colon \ast_{\cat} \rightarrow X$ and call $x$ an \textit{element of $X$} denoted by $x\in X$.
\end{remark}
\begin{proposition}
Let $\cat$ be an $\infty$-category and let $\cat'$ be the full subcategory of final objects of $\cat$. Then $\cat'$ is either empty or a contractible Kan complex.
\end{proposition}
\begin{proof}
See \cite[Proposition 1.2.12.9.]{Lurie2009} and notice the equivalence of our definition of a final object with the definition provided in \cite[Definition 1.2.12.1.]{Lurie2009} by \cite[Proposition 1.2.12.4., Corollary 1.2.12.5. and Corollary 4.2.1.8.]{Lurie2009}.
\end{proof}
By the proposition, we can assume the final object of an $\infty$-category to be unique (up to a contractible choice). In the ordinary categorical setting, there is a similar statement for final objects. By the next remark, the notions are connected in a natural way just as one wishes for an appropriate definition.
\begin{remark}
For an ordinary category $\cat$ an object is final in the nerve $N(\cat)$ if and only if it is final in $\cat$ in the usual sense by \cite[Example 1.2.12.7.]{Lurie2009} .
\end{remark}
\begin{example}{\textbf{Final anima}}\\
The (up to a contractible choice unique) final object $\ast_{\Ani}$ in the $\infty$-category $\Ani$ is given by the final Kan complex which itself corresponds to the final CW complex, i.e. it is given by the $0$-standard simplex $\triangle^0=\{\ast\}$. An element $a$ of an anima $A$ can be identified with a map $a\colon \triangle^{0} \rightarrow A$ via the Yoneda embedding. This justifies our simplification of Remark \ref{element}.
\end{example}
\begin{proposition}\label{properties}
Let $\cat$ be an $\infty$-category and $\cat_{*}$ its $\infty$-category of pointed objects. Then, the following is true:
\begin{itemize}
\item[1.] If $\cat$ is accessible, $\cat_{*}$ is accessible as well.
\item[2.] The limit over a diagram $\{\ast_{\cat} \rightarrow C_i\}_{i \in I}$ exists in $\cat_{*}$ whenever the limit over the induced diagram $\{C_i\}_{i \in I}$ in $\cat$ exists and it coincides with the unique morphism $\ast_{\cat} \rightarrow \lim_{i \in I} C_i$ induced by the universal property of the limit.
\end{itemize}
\end{proposition}
\begin{proof}
By \cite[Lemma 7.2.2.8., Proposition 4.2.1.5.]{Lurie2009}, the category $\cat_{*}$ can be identified with the under-$\infty$-category $\cat_{*/}$. Thus, the first statement on accessibility follows by \cite[Corollary 5.4.5.16.]{Lurie2009}. The second statement follows by the general fact \cite[Corollary 5.1.2.3.]{Lurie2009} that limits in functor-$\infty$-categories can be computed objectwise .
\end{proof}
\begin{definition}{\textbf{Morphisms of pointed $\infty$-categories}}\label{morpoint}\\
Let $F\colon \cat \rightarrow \mathcal{D}$ be a morphism of $\infty$-categories that  preserves final objects, i.e. the image $F(\ast_{\cat})$ can be identified with the final object $\ast_{\mathcal{D}} \in \mathcal{D}_{*}$. Then, one can define a morphism $F_{\ast}\colon \cat_{\ast}\rightarrow\mathcal{D}_{\ast}$ of pointed $\infty$-categories which sends an object $\ast_{\cat} \rightarrow X\in \cat_{*}$ to an object $\ast_{\mathcal{D}} \rightarrow F(X)\in \mathcal{D}_{\ast}$.
\end{definition}
\begin{proposition}{\textbf{Final objects in functor categories}}\label{finalobj}\\
Let $\ast_{\cat}$ be the final object in an $\infty$-category $\cat$. Then, one has the following:
\begin{itemize}
\item[1.] If the $\infty$-category $\Pro(\cat)$ is defined, it has a final object given by the corepresentable functor $\Hom_{\cat}(\ast_{\cat}, \cdot)$.
\item[2.] For every $\infty$-category $\mathcal{D}$,  the $\infty$-category $\Fun(\mathcal{D}, \cat)$ has a final object given by the functor with constant image $\ast_{\cat}$.
\end{itemize}
\end{proposition}
\begin{proof}
\begin{itemize}
\item[1.] By the formula for morphisms in a pro-$\infty$-category, one has  for every $Y=\{Y_i\}_{i \in I}\in \Pro(\cat)$ the following identifications
\begin{align*}
\Hom_{\Pro(\cat)}(Y, \Hom_{\cat}(\ast_{\cat}, \cdot))&=\colim_{i \in I^{\mathrm{op}}}\Hom_{\cat}(Y_i, \ast_{\cat})\\
&=\colim_{i \in I^{\mathrm{op}}} \ast_{\Ani}= \ast_{\Ani}
\end{align*}
where we used for the second equality the property of $\ast_{\cat}$ to be final in $\cat$.
\item[2.] Let $c_{\ast}$ denote the functor with constant image $\ast_{\cat}$ and $F$ another functor in $\Fun(\mathcal{D},\cat)$. For all $X\in \mathcal{D}$ the hom-spaces
$$\Hom_{\cat}(F(X), \ast_{\cat})$$ are contractible by definition of $\ast_{\cat}$. Thus, the same holds for $$\Hom_{\Fun(\mathcal{D},\cat)}(F, c_{\ast})$$ and $c_{*}$ is final in $\Fun(\mathcal{D}, \cat)$.
\end{itemize}
\end{proof}
\section{Homotopy pro-groups for \texorpdfstring{$\CondAni$}{condensed anima} }\label{groups}
As stated in Chapter $4$, one can define homotopy groups for animated sets by the construction of simplicial homotopy groups (cf.\ Def.\ \ref{simphomgr}). In order to define homotopy groups for condensed sets viewed as condensed anima, we can use the results of the preceding section to extend the construction of simplicial homotopy groups first to pro-anima and then to condensed anima as we will see in the following.
\begin{lemma}\label{proright}
There exists a unique functor $$R_{\Ani}\colon \Pro(\Ani) \rightarrow \CondAni$$ preserving small cofiltered limits in $\Pro(\Ani)$, which sends a pro-anima $A$ to a condensed anima given by the assignment
\begin{align}\label{form2}
X=\{X_i\}_{i \in I} \mapsto \colim_{i \in I^{\mathrm{op}}}\Hom_{\Pro(\Ani)}(v_{\Pro(\Ani)}(X_i),A),
\end{align}
where the colimit is filtered and the composition of the canonical embeddings $$v_{\Pro(\Ani)}\coloneqq j_{\Ani}\circ v_{\Ani}\colon \SFin \hookrightarrow \Ani \hookrightarrow \Pro(\Ani)$$ is fully faithful. The condensed anima defined by the formula (\ref{form2}) preserves cofiltered limits in $\Profin$ and the restriction of $R_{\Ani}$ to $\Ani$ coincides with the discrete functor $\underline{(\cdot)}^{\mathrm{disc}}$.
\end{lemma}
\begin{proof}
Consider the $\infty$-category $\CondAni\simeq \AniCond$ of condensed anima and the $\infty$-category $\Ani$ of animated sets. Both categories are presentable and complete by \ref{presen}. In particular, they are accessible $\infty$-categories that admit finite limits. Hence, the hypotheses of Proposition \ref{equiv} are satisfied.\\
Let $\underline{(\cdot)}^{\mathrm{disc}}\colon \Ani \rightarrow \CondAni$ be the discrete functor that assigns to every animated set $A$ the discrete condensed anima $\underline{A}^{\mathrm{disc}}$ as in \ref{discr}. Then, by the universal property of pro-$\infty$-objects, it extends to a unique functor $$R_{\Ani}\colon \Pro(\Ani) \rightarrow \CondAni$$ preserving small cofiltered limits. The mapping $\Pro(\Ani) \rightarrow \mathcal{PS}\mathrm{h}(\extdis)$ given by sending a pro-anima $A$ to the presheaf defined by (\ref{form2}) preserves cofiltered limits as the right hand side given by $\Hom_{\Pro(\Ani)}(v_{\Pro(\Ani)}(X_i),A)$ commutes with limits in $A$. By fully faithfulness of the Yoneda embedding $j_{\Ani}$, it clearly extends the discrete functor. If we can argue that (\ref{form2}) defines a condensed anima, the claim follows from the uniqueness of the extension $R_{\Ani}$. As $v_{\Ani}$ and $j_{\Ani}$  preserve finite coproducts (cf. Cor. \ref{collev}), the same is true for their composition $v_{\Pro(\Ani)}$. Then, we can conclude by the same argument as in Proposition \ref{constsheaf}.
\end{proof}
Note that the condensed anima corresponding to some pro-anima under the functor $R_{\Ani}$ is the same as one obtains by the construction of Lemma \ref{constinC} and Proposition \ref{constsheaf} which we also used to define discrete condensed anima.
\begin{proposition}\label{proleft}
There exists a left adjoint functor to $R_{\Ani}$ denoted by $$L_{\Ani}\colon \CondAni \rightarrow \Pro(\Ani)$$ which preserves colimits. For every discrete condensed anima $\underline{A}^{\mathrm{disc}}$ the corresponding pro-anima coincides with the corepresentable pro-anima $\Hom_{\Ani}(A,\cdot)\colon \Ani \rightarrow \Ani$ obtained by the Yoneda embedding $j_{\Ani}\colon \Ani \rightarrow \Pro(\Ani)$.
\end{proposition}
\begin{proof}
With the same argument as in the preceding lemma, the preliminaries of Proposition \ref{proext} are satisfied. Thus, there is a pro-existent left adjoint to $R_{\Ani}$ given by the functor
\begin{align*}
L_{\Ani}\colon \CondAni \rightarrow \Pro(\Ani) 
\end{align*}
sending a condensed anima $C$ to the assignment
$$B\mapsto \Hom_{\CondAni}(C,\underline{B}^{\mathrm{disc}})$$ with $B\in \Ani$. The functor $L_{\Ani}$ preserves colimits as it is a left adjoint. If $C$ is a discrete condensed anima given by $C=\underline{A}^{\mathrm{disc}}$ for some anima $A$, this reduces to the assignment
$$B\mapsto \Hom_{\Ani}(A,B)$$
by fully faithfulness of the discrete functor.
\end{proof}
\begin{corollary}
The final object in $\CondAni$ can be represented by the discrete condensed anima
\begin{align}\label{ass}\{X_i\}_{i \in I}\mapsto \colim_{i \in I^{\mathrm{op}}}\Hom_{\Ani}(v_{\Ani}(X_i), \ast_{\Ani})\cong \ast_{\Ani}
\end{align}
corresponding to $\ast_{\Ani}$ and the functor $R_{\Ani}\colon \Pro(\Ani)\rightarrow \CondAni$ preserves final objects. 
\end{corollary}
\begin{proof}
As $\ast_{\Ani}$ is the final object in $\Ani$, for every $X=\{X_i\}_{i \in I} \in \extdis$ the assignment (\ref{ass}) reduces to the functor $c_{\ast}\in \Fun(\extdis^{\mathrm{op}}, \Ani)$ with constant image $\ast_{\Ani}$. 
By the Proposition \ref{finalobj} $(2.)$ and as $\CondAni$ is a full subcategory of $\Fun(\extdis^{\mathrm{op}}, \Ani)$, the functor $c_{\ast}\in \CondAni$ is also final in $\CondAni$ and can be represented by the discrete condensed anima given by (\ref{ass}). As the final object in $\Pro(\Ani)$ is simply given as the image of the final object in $\Ani$ under $j_{\Ani}$, its image under $R_{\Ani}$ is by Lemma $\ref{proright}$ given by the discrete condensed anima corresponding to $\ast_{\Ani}$ and the claim follows.
\end{proof}
\begin{corollary}
The functor $L_{\Ani}\colon \CondAni \rightarrow \Pro(\Ani)$ preserves final objects. Thus, we obtain a functor
$$(L_{\Ani})_{\ast}\colon \CondAni_{\ast} \rightarrow \Pro(\Ani)_{*}$$ of pointed $\infty$-categories given by the assignment
$$c \mapsto L_{\Ani}(c),$$
where $ L_{\Ani}(c)$ denotes the image of the morphism $c\colon \ast_{\CondAni} \rightarrow C$ under $L_{\Ani}$.
\end{corollary}
\begin{proof}
By the preceding corollary and Proposition \ref{proleft}, the image of the final object in $\CondAni$ under the left adjoint $L_{\Ani}$ is given by the copresentable functor $\Hom_{\Ani}(\ast,\cdot) \in \Pro(\Ani)$, which is exactly the final object in the pro-$\infty$-category by Proposition \ref{finalobj}. Thus, the functor $(L_{\Ani})_{\ast}$ can be defined as described in Definition \ref{morpoint}.
\end{proof}
\begin{remark}
The preceding results can be summarized as follows: One obtains a commutative diagram 
\begin{equation*}
\begin{tikzcd}[row sep=huge]
&
\Ani \arrow[hookrightarrow,dl,swap, "\underline{(\cdot)}^{\mathrm{disc}}"] \arrow[hookrightarrow,dr, "j_{\Ani}"] & 
\\ 
\CondAni \arrow[rr, yshift=0.7ex, "L_{\Ani}"] & & \Pro(\Ani) \arrow[ll, yshift=-0.7ex, "R_{Ani}"]
\end{tikzcd}
\end{equation*}
which can be extended to a commutative diagram on the corresponding pointed $\infty$-categories as all involved functors preserve final objects.
\end{remark}
\begin{proposition}
If the pro-$\infty$-category $\Pro(\cat)$ is defined, then the $\infty$-category $\Pro(\cat_{*})$ is defined as well and one has an equivalence of $\infty$-categories
$$\Pro(\cat_{*})=\Pro(\cat)_{*}$$
induced by the map $f\colon\Fun(\triangle^1, \cat)\rightarrow \Fun(\triangle^1, \Pro(\cat))$, which is itself induced by the Yoneda embedding $j_{\cat}\colon \cat \rightarrow \Pro(\cat)$.
\end{proposition}
\begin{proof}
The $\infty$-category $\mathcal{D}_{*}$ admits finite limits and is accessible whenever this is true for $\mathcal{D}$ by Proposition \ref{properties}. This proves the first claim. As $j_{\cat}$ preserves final objects, the same is true for the functor $f$. Thus, the restriction $f'$ of $f$ to $\cat_{*}$ lands in $\Pro(\cat)_{*}$. This restriction $f'$ can by Proposition \ref{equiv} be extended to a unique functor 
$$\Pro(f')\colon\Pro(\cat_{*})\rightarrow \Pro(\cat)_{*},$$
preserving cofiltered limits as $\Pro(\cat)_{*}$ admits cofiltered limits by Proposition \ref{properties}.
We want to show that this indeed defines an equivalence of $\infty$-categories.
First, we note that an object in $\Pro(\cat)_{*}$ is given as a morphism $$\ast \rightarrow \mathrm{lim}_{i \in I}j_{\cat}(X_i),$$ where $\mathrm{lim}_{i \in I}j_{\cat}(X_i)$ is a pro-object in $\cat$ represented as a cofiltered limit as usual. This coincides with $\mathrm{lim}_{i \in I}(\ast \rightarrow j_{\cat}(X_i))$ by Proposition \ref{properties}. Thus, every object in $\Pro(\cat)_{*}$ is a cofiltered limit of objects $\ast\rightarrow j_{\cat}(X_i)$. As we know that objects in $\Pro(\cat_{*})$ can be written as cofiltered limits of objects $j_{\cat_{*}}(\ast'\rightarrow X_i)$, where $\ast'$ is the final object in $\cat_{*}$, it is enough to show the equivalence of the full subcategories given by the objects $\ast\rightarrow j_{\cat}(X_i)$ in $\Pro(\cat)_{*}$ and the objects $j_{\cat_{*}}(\ast'\rightarrow X_i)$ in $\Pro(\cat_{*})$. But this is clear, as the subcategory of objects $j_{\cat_{*}}(\ast'\rightarrow X_i)$ in $\Pro(\cat_{*})$ is equivalent to $\cat_{*}$ by fully faithfulness of $j_{\cat_{*}}$ and $\cat_{*}$ itself is equivalent to the subcategory of objects $\ast \rightarrow j_{\cat}(X_i)$ by fully faithfulness of the functor $j_{\cat}$.
\end{proof}
Now, we have collected all the necessary components to extend the notion of homotopy groups to pro-$\infty$-categories and to use this to define homotopy pro-groups for condensed anima.
\begin{proposition}\label{progroups}
There exist unique functors
\begin{align*}
&\Pro(\pi_0)\colon \Pro(\Ani) \ \rightarrow \Pro(\Set)\\
&\Pro(\pi_1)\colon \Pro(\Ani)_{*} \rightarrow \Pro(\Grp)\\
&\Pro(\pi_n) \colon \Pro(\Ani)_{*} \rightarrow \Pro(\AbGrp)
\end{align*}
with $n\geq 2$ preserving small cofiltered limits, s.t. for every anima $A$ regarded as an object in $\Pro(\Ani)$ and $a \in A$ the image coincides with the corresponding simplicial homotopy group $\pi_n(A,a)$ or the set of path components $\pi_0(A)$, respectively.
\end{proposition}
\begin{proof}
One has functors
\begin{align*}
&\pi_0\colon \Ani \ \rightarrow \Set\\
&\pi_1\colon \Ani_{*} \rightarrow \Grp\\
&\pi_n \colon \Ani_{*} \rightarrow \AbGrp
\end{align*}
with $n\geq 2$ by Definition \ref{simphomgr} and Remark \ref{simphomgrre}. As all included $\infty$-categories are accessible and admit finite limits by Proposition \ref{properties}, we can apply Corollary \ref{funcacc} to obtain the corresponding pro-functors with the claimed properties by using the identification $\Pro(\Ani_{\ast})=\Pro(\Ani)_{\ast}$ which is true by the preceding proposition. 
\end{proof}
\begin{definition}{\textbf{Homotopy pro-groups for condensed anima}}\label{homo}\\
For every condensed anima $C\in \CondAni$ we define the pro-set of path components by $$\widetilde{\pi}_0(C)\coloneqq \pi_0(L_{\Ani}(C))$$ and for some $c \in C$
the $n$th homotopy pro-group for $n\geq 1$ by $$\widetilde{\pi}_n(C,c)\coloneqq \pi_n((L_{\Ani})_{\ast}(c)).$$
\end{definition}
We end this section by showing that our definition of homotopy pro-groups is a proper extension of the notion of homotopy groups on topological spaces in the sense that it deals with the class of CW-spaces properly.
\begin{proposition}\label{CW}
For any CW-space $X$, the homotopy pro-groups of the condensed anima $C\coloneqq \Hom_{\Top}(\cdot,X)$ represented by $X$ coincide with the usual topological homotopy groups.
\end{proposition}
\begin{proof}
Recall the adjunction
\[
 \begin{tikzcd}[column sep = huge]
            \sSet \arrow[r, shift left=1ex, "|\cdot|"] & \Top, \arrow[l, shift left=.5ex, "\mathrm{Sing}(\cdot)"]
           \end{tikzcd}
    \]
which induces an equivalence on the (homotopy) categories of of CW-spaces and of Kan complexes. Let $A$ denote the anima corresponding to a CW-space $X$ through the (non fully faithful) map $\CW \rightarrow \Ani$ induced by $\mathrm{Sing}(\cdot)$. We want to conclude that our left adjoint $L_{\Ani}$ restricted to CW-spaces coincides with the map sending the condensed anima $C\coloneqq \Hom_{\Top}(\cdot,X)$ corresponding to X to the anima $A$. By Lemma 11.9. in \cite{Scholze2019a}, there is a natural map $C \rightarrow \underline{A}^{\mathrm{disc}}$ of condensed anima which is universal in this sense. Thus, one can deduce the following chain of equivalences 
$$\Hom_{\Ani}(A,B)\cong\Hom_{\CondAni}(\underline{A}^{\mathrm{disc}}, \underline{B}^{\mathrm{disc}})\cong\Hom_{\CondAni}(C, \underline{B}^{\mathrm{disc}})$$
by fully faithfulness of the discrete functor $\underline{(\cdot)}^{\mathrm{disc}}\colon \Ani \rightarrow \CondAni$ and universality of the map $C\rightarrow \underline{A}^{\mathrm{disc}}$. This directly implies left adjointness of the functor $\widetilde{L}\colon\CondAni \supset\CW \rightarrow \Ani, \ C \mapsto A$ to the discrete functor. By uniqueness of adjoint functors, the functor $\widetilde{L}$ has to coincide with the restriction of our left adjoint $L_{\Ani}$ to CW-spaces. Hence, as stated in Proposition \ref{progroups}, the homotopy pro-groups of $C$ defined in Definition \ref{homo} coincide with the simplicial homotopy groups of $A$. But, by Remark \ref{simphomgrre} these groups coincide with the topological homotopy groups of $X$.
\end{proof}
\section{Profinite sets as an example}\label{example}
In the proof of Proposition \ref{CW}, we have seen that there is a partially left adjoint to $\Ani \hookrightarrow \CondAni$ defined on those condensed anima which are represented by CW-spaces. This partially left adjoint induced the equivalence of the different notions of homotopy (pro-)groups on CW-spaces. In this last section, we want to determine the homotopy pro-groups of all condensed anima that can be represented by a profinite set. These illustrate a class of condensed anima justifying our approach of introducing pro-anima. Indeed, we will see that $\Ani \hookrightarrow \CondAni$ does not preserve limits and hence cannot admit a left adjoint. This fact is stated in more detail in Remark \ref{noleftadj}. 
\begin{proposition}\label{notdisc}
For every condensed anima represented by a profinite set, the corresponding pro-anima under the functor $L_{\Ani}$ of Proposition \ref{proleft} is given by the profinite set itself and thus corresponds to an anima if and only if
the profinite set is finite. 
\end{proposition}
\begin{proof}
The pro-anima corresponding to a profinite set $S=\mathrm{lim}_{i \in I} S_i$ written as small cofiltered limit of finite (discrete) sets is given as a functor 
\begin{align}\label{eq22}
\Ani \rightarrow \Ani, \ A\mapsto\mathrm{lim}_{i \in I}\Hom_{\Ani}(v(S_{i}),A)
\end{align} in $\Pro(\Ani)$ by Proposition \ref{proaniprofin}.\\
On the other hand, the image of the condensed anima $S^{*}\coloneqq\Hom_{\Top}(\cdot, S)\in \CondAni$ represented by the profinite set $S$ under the pro-existent left adjoint $L_{\Ani}\colon \CondAni \rightarrow \Pro(\Ani)$ of Proposition \ref{proleft} is given by the pro-anima 
\begin{align}\label{eq3}
\Ani \rightarrow \Ani, \ A\mapsto\Hom_{\CondAni}(S^{*}, \underline{A}^{\mathrm{disc}}).
\end{align}
In the following, it is our goal to show that the functors (\ref{eq3}) and (\ref{eq22}) coincide.\\
We can view $S^{*}$ and $\underline{A}^{\mathrm{disc}}$ as presheaves from all profinite sets to anima. This is possible as both condensed anima are given as a restriction of the contravariant functor that was defined on all profinite sets by the same formula $\Hom_{\Top}(\cdot,S)$ (cf.\ Cor.\ \ref{exttoall}) or $\colim_{i \in I^{\mathrm{op}}}\Hom_{\Ani}(v(\cdot),A)$ (cf.\ Prop.\ \ref{constsheaf}), respectively. For $\underline{A}^{\mathrm{disc}}$ we denoted this functor by $A^{\mathrm{disc}}$ in Lemma \ref{constinC}. Then, one can conclude by the $\infty$-categorical Yoneda Lemma
\begin{align*}
\Hom_{\mathcal{PS}\mathrm{h}(\Pro(\Set^{\mathrm{fin}}))}(\Hom_{\Top}(\cdot, S), A^{\mathrm{disc}})
\simeq & A^{\mathrm{disc}}(S)
\end{align*}
and moreover
\begin{align*}
 A^{\mathrm{disc}}(S)
= \colim_{i \in I^{\mathrm{op}}}\Hom_{\Ani}(v(S_i),A)
\end{align*}
as in Lemma \ref{constinC}.
We obtain a functor 
\begin{align*}
\Ani \rightarrow \Ani, \ A\mapsto A^{\mathrm{disc}}(S)
\end{align*} in $\Pro(\Ani)^{\mathrm{op}}$. Thus, by regarding it as a functor in $\Pro(\Ani)$, i.e. turning the filtered colimit into a cofiltered limit, one can see that the functor  (\ref{eq3}) is the same as (\ref{eq22}).
\end{proof}
\begin{remark}\label{noleftadj}
In Proposition \ref{discfin}, we have seen that the limit $\mathrm{lim}_{i \in I} \underline{X_i}^{\mathrm{disc}}$ of discrete condensed anima defined by finite sets $X_i$ is given by $\Hom_{\Top}(\cdot, \mathrm{lim}_{i \in I} X_i)$, i.e. coincides with the condensed anima represented by the profinite set $\mathrm{lim}_{i \in I} X_i$. By Proposition \ref{notdisc},
this condensed anima is not discrete if the profinite set is not a finite set itself. Hence, the fully faithful embedding $\Ani \hookrightarrow \CondAni$ does not preserve cofiltered limits. As every right adjoint preserves limits, one cannot construct a left adjoint to the embedding. This justifies our approach of passing to the category of pro-anima.
\end{remark}
\begin{corollary}
For every condensed anima represented by a profinite set the pro-set of path components is given by the profinite set itself and all homotopy pro-groups are trivial.
\end{corollary}
\begin{proof}
Let $X=\{X_i\}_{i \in I}$ be the profinite set representing the condensed anima. By Proposition \ref{notdisc} and Definition \ref{homo} the statement reduces to computing the images of the (pointed) profinite set regarded as (pointed) pro-anima $S=\{S_i\}_{i \in I}$ under the functors of Proposition \ref{progroups}, i.e. the anima $S_i$ referring to the finite sets $X_i$. Since the functors $\Pro(\pi_n)$ preserve small cofiltered limits, one can identify the image of the profinite set $S$, respectively of the pointed profinite set $(S,s)$, under $\Pro(\pi_0)$, respectively some $\Pro(\pi_n)$ for $n\geq 1$, with a small cofiltered limit of the images $\Pro(\pi_0(S_i))$, respectively $\Pro(\pi_n(S_i, s_i))$, where $s_i$ denotes the image of $s$ under the natural map $S\rightarrow S_i$. Moreover, Proposition \ref{progroups} states that for every anima $S_i$ these images coincide with the simplicial homotopy groups $\pi_0(S_i)$, $\pi_n(S_i,s_i)$ which in turn coincide with the topological homotopy groups of the corresponding finite set $X_i$ considered as finite discrete space. 
Thus, the set $\pi_0(S_i)$ of path components is the finite set $X_i$ and every homotopy group $\pi_n(S_i,s_i)$ is trivial. Altogether, the set of path components of $S$ is the cofiltered limit of the finite sets $X_i$ and thus the profinite set $X$ itself. The homotopy pro-groups are trivial, as limits of trivial groups are trivial.
\end{proof}
\addcontentsline{toc}{chapter}{References}
%
%
\nocite{*}
\printbibliography[title={References}]
%
%
\end{document}